\documentclass[11pt]{amsart}
\usepackage{amsmath,amssymb,amsthm,mathtools}
\usepackage[margin=1.15in]{geometry}
\usepackage{booktabs,longtable,array}
\usepackage{microtype}
\usepackage[hidelinks]{hyperref}
\usepackage{enumitem}
\setlist{leftmargin=2em}

\newcommand{\repo}{\url{https://github.com/jaideepsaipadhi/the-auluck-chowla-gupta-conjecture} (archived as \href{https://doi.org/10.5281/zenodo.23241759}{doi:10.5281/zenodo.23241759})}
\newcommand{\script}[1]{\texttt{#1}}

\newtheorem{theorem}{Theorem}[section]
\newtheorem{lemma}[theorem]{Lemma}
\newtheorem{proposition}[theorem]{Proposition}
\newtheorem{corollary}[theorem]{Corollary}
\newtheorem*{maintheorem}{Main Theorem}
\theoremstyle{definition}

\theoremstyle{remark}
\newtheorem{remark}[theorem]{Remark}

\numberwithin{equation}{section}

\newcommand{\RR}{\mathbb{R}}

\newcommand{\ZZ}{\mathbb{Z}}
\newcommand{\Li}{\operatorname{Li}}
\newcommand{\TV}{\operatorname{TV}}
\newcommand{\Var}{\operatorname{Var}}
\newcommand{\erfc}{\operatorname{erfc}}
\newcommand{\sinc}{\operatorname{sinc}}
\newcommand{\re}{\operatorname{Re}}
\newcommand{\im}{\operatorname{Im}}
\newcommand{\E}{\mathbb{E}}
\newcommand{\Prob}{\mathbb{P}}
\newcommand{\dd}{\,\mathrm{d}}
\newcommand{\ii}{\mathrm{i}}

\title[The Auluck--Chowla--Gupta conjecture]{A proof of the Auluck--Chowla--Gupta Conjecture}
\author{Jaideep Sai Padhi}
\address{Purdue University, West Lafayette, IN 47907, USA}
\email{jpadhi@purdue.edu}
\subjclass[2020]{Primary 05A17; Secondary 11P82, 05A20, 65G30}
\keywords{integer partitions, unimodality, saddle-point method, interval arithmetic, computer-assisted proof}

\begin{document}

\begin{abstract}
Let $p(n,k)$ denote the number of partitions of $n$ into exactly $k$ parts. In 1942 Auluck, Chowla and Gupta
conjectured that for every $n$ the sequence $p(n,1),p(n,2),\dots,p(n,n)$ is unimodal. Szekeres proved this for all
sufficiently large $n$, but with a threshold that was never made explicit. We prove the conjecture for every $n\ge 1$.
For $n\ge 10^5$ we give an analytic argument with fully explicit constants: an exact probabilistic representation of the
ratio $p(n,k+1)/p(n,k)$, a second-order saddle-point main term with a rigorous uniform error bound in the transition window,
and separate elementary or saddle-point estimates in the two tails. Every numerical constant in this argument is certified in
ball arithmetic. For $n\le 2\cdot10^5$ the conjecture is verified exhaustively by two independent programs, one of them using
exact integer arithmetic only. Since log-concavity of $k\mapsto p(n,k)$ fails for every $n\ge 4$, the argument works with
first differences and never uses log-concavity.
\end{abstract}

\maketitle

%======================================================================
\section{Introduction}\label{sec:intro}
%======================================================================

For integers $n\ge1$ and $1\le k\le n$ let $p(n,k)$ be the number of partitions of $n$ into exactly $k$ positive parts. A finite
sequence $a_1,\dots,a_N$ of real numbers is \emph{(weakly) unimodal} if there is an index $t$ with
$a_1\le\dots\le a_t\ge a_{t+1}\ge\dots\ge a_N$.

\begin{maintheorem}
For every integer $n\ge1$ the sequence $p(n,1),p(n,2),\dots,p(n,n)$ is unimodal.
\end{maintheorem}

\noindent In fact we prove the strict form, in the shape of Szekeres' theorem, for every $n$
(Theorem~\ref{thm:strict} in Section~\ref{sec:finite}): for $n\ge4$ there is $t$ with
\[
  p(n,1)<\dots<p(n,t)\ \ge\ p(n,t+1)>p(n,t+2)>\dots>p(n,n-1)=p(n,n),
\]
and $p(n,t)=p(n,t+1)$ only for $n\in\{5,6,8,13,14\}$.

\subsection*{History}
The statement was conjectured by Auluck, Chowla and Gupta~\cite{ACG42}, who also showed that the mode $k_0(n)$ satisfies
$n^{1/2}<k_0\le(1+\delta)c^{-1}n^{1/2}\log n$ for large $n$, where $c=\pi\sqrt{2/3}$. Erd\H{o}s~\cite{Erdos46} determined $k_0$ up to $o(n^{1/2})$
(see also Erd\H{o}s--Lehner~\cite{EL41}) and wrote that he was unable to prove or disprove the unimodality conjecture.
Szekeres~\cite{Sz51,Sz53} obtained an asymptotic formula for $p(n,k)$ uniform in $k$, and in~\cite[Theorem~2]{Sz53} proved that
\emph{if $n$ is sufficiently large} there is $k_1(n)$ with $p(n,k)<p(n,k+1)$ for $k<k_1$ and $p(n,k)>p(n,k+1)$ for $k>k_1$, remarking
that this proves the conjecture of Auluck, Chowla and Gupta. His constants are not made explicit, so the argument gives no value of
$n$ beyond which it applies (see also~\cite{Sz90}). Canfield~\cite{Can97} later gave a different derivation of Szekeres' formula, by
recursions rather than by contour integration. In a 2003 survey Canfield~\cite{Can03} notes that unimodality had been observed
computationally for every row with $n\le2000$, that Szekeres' theorem covers only sufficiently large $n$, and that the question is
therefore ``technically not completely resolved''; Balazard~\cite{Bal90} likewise lists it as known only for large $n$. As recently as 2015, Chubarikov and Fedorov~\cite{CF15} describe the monotonicity
conjecture of Auluck, Chowla and Gupta as unsolved. To our knowledge no proof for all $n$ has appeared before.

\subsection*{Why first differences}
Unimodality would follow from log-concavity, $p(n,k)^2\ge p(n,k-1)p(n,k+1)$, but this fails at the tail of every row with
$n\ge4$: for $n/2\le k\le n$ one has $p(n,k)=p(n-k)$ (Lemma~\ref{lem:E}), so $p(n,n)=p(n,n-1)=1$ and $p(n,n-2)=2$, and
$p(n,n-1)^2=1<2=p(n,n)p(n,n-2)$. We therefore work directly with the sign of the first difference, through the ratio
\[
  R_k=R_k(n):=\frac{p(n,k+1)}{p(n,k)}\qquad(1\le k\le n-1).
\]
The row is unimodal if and only if the signs of $\log R_k$, read left to right, never change from negative to positive.

\subsection*{Strategy}
Put
\begin{equation}\label{eq:beta}
  \beta:=\frac{\sqrt{6n}}{\pi},\qquad \ell:=\log\beta,\qquad \mu:=\beta\ell .
\end{equation}
The mode of the row lies within $O(\beta)$ of $\mu$ (this is the Gumbel scaling of Erd\H{o}s and Lehner~\cite{EL41}).
For $n\ge10^5$ we split $1\le k\le n-1$ into five ranges,
\[
\begin{gathered}
  D=[1,k_D],\qquad CL=(k_D,\mu-2\beta],\qquad W=(\mu-2\beta,\mu+2\beta),\\
  CR=[\mu+2\beta,\lceil n/2\rceil],\qquad E=[\lceil n/2\rceil,n-1],
\end{gathered}
\]
with $k_D=\lfloor1.7n^{1/3}\rfloor$, and determine the sign of $\log R_k$ on each range separately.

\begin{itemize}
\item On $D$ an elementary binomial sandwich shows $R_k>1$ (Section~\ref{sec:D}).
\item On $E$ an exact bijection shows $R_k\le1$ (Lemma~\ref{lem:E}).
\item Everything else rests on an exact identity (Section~\ref{sec:rep}): $R_k$ equals the ratio of two Fourier integrals
  of the characteristic function of a sum of independent geometric variables, taken at its saddle point $\theta$. This
  reduces the problem to estimating a ratio of integrals with the \emph{same} weight, so that all exponentially large factors
  cancel exactly.
\item On $CL$ and $CR$ a zeroth-order version of this estimate is enough to give $R_k>1$ and $R_k<1$ respectively
  (Section~\ref{sec:CLCR}).
\item In the window $W$, where the sign changes, we need more. We define an explicit second-order main term $L_{2,k}$
  (Section~\ref{sec:main}) and prove $|\log R_k-L_{2,k}|\le 0.4282\,\theta q$, where $q=e^{-(k+1)\theta}$ (Theorem~\ref{thm:A}).
  We then show that $L_2$ decreases by at least $0.91\,\theta q$ at each step (Theorem~\ref{thm:B}). Since
  $0.91>2\times0.4282$, at most one $k\in W$ has an undetermined sign, and that $k$ is adjacent to the sign change of $L_2$.
\end{itemize}

The constants in these estimates are bounded by rigorous ball arithmetic (Arb~\cite{Arb}, through python-flint) over finite
covers of the relevant parameter ranges. The analytic argument is carried out for $n\ge10^5$, and the range $n\le2\cdot10^5$
is checked exhaustively (Section~\ref{sec:finite}).

\subsection*{Computer assistance}
All computer-assisted steps are listed in Section~\ref{sec:computer} (Table~\ref{tab:certs}), each with the script that
certifies it and the exact statement certified. The scripts, their expected outputs and a driver that re-runs everything and
compares outputs are available at \repo. The mathematical arguments in this paper do not depend on any computation other than
those listed there.

\subsection*{Organisation}
Section~\ref{sec:rep} gives the notation and the exact representation. Section~\ref{sec:deriv} proves uniform derivative bounds.
Section~\ref{sec:main} defines the main term and states and proves the error bound of Theorem~\ref{thm:A}; its uniformity rests on
the ingredient bounds of Section~\ref{sec:ingredients}. Section~\ref{sec:B} analyses the step of $L_2$. Section~\ref{sec:CLCR}
treats the tail regions, Section~\ref{sec:assembly} assembles the proof for $n\ge10^5$, Section~\ref{sec:finite} describes the
finite verification, and Section~\ref{sec:computer} lists the computer-assisted inequalities and the trust base. The longer
error algebra is in the appendices.

\subsection*{Conventions}
$\zeta_2:=\zeta(2)=\pi^2/6$, so that $n=\zeta_2\beta^2$. $\Li_s(y)=\sum_{l\ge1}y^l/l^s$ is the polylogarithm; for integers
$r\ge1$, $\Li_{1-r}(y)=\sum_l l^{r-1}y^l$ is a rational function of $y$. Logarithms of complex numbers are principal and are only
taken of numbers with positive real part. For a function $g$ and an interval $I$, $\Var_I g$ is the total variation of $g$ on $I$,
and $\TV(g)=\Var_{[0,\infty)}g$.

%======================================================================
\section{Notation and an exact reduction}\label{sec:rep}
%======================================================================

Fix $n$ and $2\le k\le n-2$. Write $a_k=p(n,k)$ and
\[
  m:=n-k\ \ (\ge2),\qquad K:=k+1 .
\]
Let $f_k(N)$ be the number of partitions of $N$ into parts of size at most $k$, with $f_k(N)=0$ for $N<0$.

\begin{lemma}\label{lem:comb}
$p(n,k)=f_k(n-k)$, and $f_{k+1}(N)=\sum_{s\ge0}f_k(N-sK)$. Consequently
$a_{k+1}=\sum_{s\ge0}f_k(m-1-sK)$, a finite sum.
\end{lemma}

\begin{proof}
These are standard (see~\cite{Andrews}). The conjugate of a partition of $n$ into exactly $k$ parts has largest part exactly $k$; removing one part $k$ is a bijection
onto partitions of $n-k$ into parts $\le k$. For the second identity, classify partitions of $N$ into parts $\le K$ by the
number $s$ of parts equal to $K$ and remove them. Finally $a_{k+1}=f_{k+1}(m-1)$.
\end{proof}

\subsection{The saddle point and a random variable}
For $t>0$ let
\begin{equation}\label{eq:S}
  S_k(t):=\sum_{i=1}^{k}\frac{i}{e^{it}-1}.
\end{equation}
$S_k$ is continuous and strictly decreasing from $+\infty$ to $0$, so there is a unique $\theta=\theta(n,k)>0$ with
$S_k(\theta)=m$. We call $\theta$ the \emph{saddle point} and put
\begin{equation}\label{eq:params}
  x:=e^{-\theta},\quad q:=x^K=e^{-K\theta},\quad \nu:=K\theta,\quad \nu':=\nu+\log\theta,\quad v:=k\theta=\nu-\theta,
\end{equation}
so that $q=\theta e^{-\nu'}$. Let $Y_1,\dots,Y_k$ be independent with $\Prob(Y_i=y)=(1-x^i)x^{iy}$ $(y\ge0)$ and put
$X:=\sum_{i\le k}iY_i$.

\begin{lemma}\label{lem:X}
$\Prob(X=N)=f_k(N)x^N\prod_{i\le k}(1-x^i)$, and $\E X=S_k(\theta)=m$.
\end{lemma}

\begin{proof}
$\Prob(X=N)=\prod_i(1-x^i)\cdot x^N\cdot\#\{(y_i):\sum iy_i=N\}$ and the count is $f_k(N)$, with $y_i$ the multiplicity of the
part $i$. Also $\E Y_i=x^i/(1-x^i)=1/(e^{i\theta}-1)$.
\end{proof}

\begin{proposition}[Exact representation]\label{prop:rep}
Let $\chi(\varphi):=\E\,e^{\ii\varphi(X-m)}$ and
\[
  g_0:=\frac{x}{1-q},\qquad H(\varphi):=\ii\varphi-\log\bigl(1-qe^{\ii K\varphi}\bigr)+\log(1-q).
\]
Then
\begin{equation}\label{eq:rep1}
  R_k=\sum_{s\ge0}x^{1+sK}\,\frac{\Prob(X=m-1-sK)}{\Prob(X=m)},
\end{equation}
and, with $D:=\int_{-\pi}^{\pi}\chi(\varphi)\dd\varphi=2\pi\Prob(X=m)>0$ and $\tilde N:=\int_{-\pi}^{\pi}\chi(\varphi)\bigl(e^{H(\varphi)}-1\bigr)\dd\varphi$,
\begin{equation}\label{eq:rep2}
  \frac{R_k}{g_0}-1=\frac{\tilde N}{D}.
\end{equation}
\end{proposition}

\begin{proof}
Since $f_k(m)\ge1$ we have $\Prob(X=m)>0$. By Lemma~\ref{lem:X}, $f_k(N)=\Prob(X=N)x^{-N}/\prod(1-x^i)$, and inserting this into
Lemma~\ref{lem:comb} gives~\eqref{eq:rep1}. The series $\chi(\varphi)=\sum_N\Prob(X=N)e^{\ii\varphi(N-m)}$ converges absolutely and
uniformly, so $\frac1{2\pi}\int_{-\pi}^{\pi}\chi(\varphi)e^{\ii j\varphi}\dd\varphi=\Prob(X=m-j)$. Inserting this into~\eqref{eq:rep1} and
exchanging the sum over $s$ with the integral (dominated by $\sum_s x^{1+sK}$) gives
$R_k=\int\chi G/\int\chi$ with $G(\varphi)=\sum_s x^{1+sK}e^{\ii(1+sK)\varphi}=xe^{\ii\varphi}/(1-qe^{\ii K\varphi})$. Since
$\re(1-qe^{\ii K\varphi})\ge1-q>0$, $H$ is well defined and $e^H=G/g_0$, which gives~\eqref{eq:rep2}.
\end{proof}

\begin{remark}
Identity~\eqref{eq:rep1} holds for every $x\in(0,1)$, not only at the saddle point, because the $x$-dependence cancels. The
saddle point is chosen so that $\chi$ is concentrated near $\varphi=0$. The form~\eqref{eq:rep2}, rather than $R_k=N/D$, matters:
a relative error $\delta$ in $D$ costs only $\delta\,|R_k/g_0-1|$, which is of order $\delta\theta^2$.
\end{remark}

By independence,
\begin{equation}\label{eq:psi}
  \chi=e^{\psi},\qquad \psi(\varphi):=-\ii m\varphi+\sum_{i\le k}\Bigl[\log(1-x^i)-\log\bigl(1-x^ie^{\ii i\varphi}\bigr)\Bigr],
\end{equation}
a real-analytic function of $\varphi\in\RR$. Since $X$ is real-valued, $\chi(-\varphi)=\overline{\chi(\varphi)}$, so
$\re\psi$ is even and $\im\psi$ is odd. Moreover
\begin{equation}\label{eq:modchi}
  -2\log|\chi(\varphi)|=S(\varphi):=\sum_{i\le k}\log\Bigl(1+\frac{\sin^2(i\varphi/2)}{\sinh^2(i\theta/2)}\Bigr),
\end{equation}
because $|1-ye^{\ii a}|^2/(1-y)^2=1+\sin^2(a/2)/\sinh^2(u/2)$ for $y=e^{-u}$. Every summand of $S$ is nonnegative.

\subsection{Cumulants and the functions \texorpdfstring{$f_r$, $F_r$}{f\_r, F\_r}}
For $r\ge2$ put
\begin{equation}\label{eq:kappa}
  \kappa_r:=\sum_{i\le k}i^r\Li_{1-r}(x^i),\qquad B:=\kappa_2,\qquad
  c_1:=1+\frac{Kq}{1-q},\qquad c_r:=K^r\Li_{1-r}(q)\ (r\ge2).
\end{equation}
For $r\ge1$ let $f_r(u):=u^r\Li_{1-r}(e^{-u})$ for $u>0$, extended continuously by $f_r(0)=(r-1)!$, and let
$F_r(v):=\int_0^vf_r$. Thus $f_1(u)=u/(e^u-1)$ and $f_2(u)=(u/2)^2/\sinh^2(u/2)$; both are decreasing and take values in $(0,1]$.
We use the scaled quantities
\begin{equation}\label{eq:scaled}
  s_r:=\theta^{r+1}\kappa_r=\theta\sum_{i\le k}f_r(i\theta),\qquad b:=s_2=\theta^3B,\qquad c_r\theta^r=f_r(\nu)\ (r\ge2),
  \qquad c_1\theta=\theta+f_1(\nu).
\end{equation}

\begin{lemma}[Closed form of $F_r$]\label{lem:Fr}
For $r\ge1$ and $v>0$,
\[
  F_r(v)=r!\Bigl(\zeta_2-\sum_{j=0}^{r}\frac{v^j}{j!}\Li_{2-j}(e^{-v})\Bigr),\qquad F_r(\infty)=r!\,\zeta_2 .
\]
In particular $F_r$ is nondecreasing and bounded by $r!\zeta_2$.
\end{lemma}

\begin{proof}
Call the right side $G_r$. From $\frac{d}{dv}\Li_s(e^{-v})=-\Li_{s-1}(e^{-v})$ and the Leibniz rule, the derivative of the sum over $j$
telescopes to $-v^r\Li_{1-r}(e^{-v})/r!$, so $G_r'=f_r$. As $v\to0^+$, $\Li_2(e^{-v})\to\zeta_2$, $v\Li_1(e^{-v})=-v\log(1-e^{-v})\to0$,
and $v^j\Li_{2-j}(e^{-v})=O(v)$ for $j\ge2$ because $\Li_{2-j}(e^{-v})=O(v^{1-j})$; hence $G_r(0^+)=0=F_r(0)$. Since $f_r$ is continuous
on $[0,v]$, $F_r=G_r$. As $v\to\infty$ every $v^j\Li_{2-j}(e^{-v})$ tends to $0$.
\end{proof}

\subsection{Main term}
Define
\begin{equation}\label{eq:M2}
\begin{aligned}
  N_0&:=1+\frac{\kappa_4}{8B^2}-\frac{5\kappa_3^2}{24B^3},\qquad
  E_2:=\frac1B\Bigl(1+\frac{\kappa_4}{2B^2}-\frac{5\kappa_3^2}{4B^3}\Bigr),\\
  E_1&:=\frac1{N_0}\Bigl(\frac{\kappa_3}{2B^2}-\frac{\kappa_5}{8B^3}+\frac{35\,\kappa_3\kappa_4}{48\,B^4}-\frac{35\,\kappa_3^3}{48\,B^5}\Bigr),\\
  M_2&:=1+c_1E_1-\frac{(c_2+c_1^2)E_2}{2}-\frac{5c_3\kappa_3}{12B^3}+\frac{c_4}{8B^2},
\end{aligned}
\end{equation}
and
\begin{equation}\label{eq:L2}
  L_k:=\log R_k,\qquad L_{2,k}:=-\theta-\log(1-q)+\log M_2 = \log g_0+\log M_2 .
\end{equation}
The origin of $M_2$ is explained in Section~\ref{sec:main}: it is the closed form of the ratio of the Gaussian model integrals
up to the order needed. Note that $L_k-L_{2,k}=\log\bigl(R_k/(g_0M_2)\bigr)$.

%======================================================================
\section{Uniform derivative bounds}\label{sec:deriv}
%======================================================================

\begin{lemma}\label{lem:psider}
$\psi(0)=\psi'(0)=0$, $\psi^{(r)}(0)=\ii^r\kappa_r$ and $|\psi^{(r)}(\varphi)|\le\kappa_r$ for all $\varphi\in\RR$ and $r\ge2$.
\end{lemma}

\begin{proof}
For $0<y<1$ and $a>0$, $\ell_{y,a}(\varphi):=-\log(1-ye^{\ii a\varphi})=\sum_{l\ge1}y^le^{\ii la\varphi}/l$, absolutely and uniformly.
The $r$-times differentiated series $(\ii a)^r\sum_ll^{r-1}y^le^{\ii la\varphi}$ is dominated by $a^r\Li_{1-r}(y)<\infty$, so by the
Weierstrass M-test termwise differentiation is valid and $|\ell_{y,a}^{(r)}|\le a^r\Li_{1-r}(y)$. Now
$\psi=-\ii m\varphi+\sum_i[\log(1-x^i)+\ell_{x^i,i}]$; for $r\ge2$ the linear part disappears, which gives the bound and the value at $0$.
Finally $\psi'(0)=\ii\bigl(-m+\sum_i ix^i/(1-x^i)\bigr)=0$ by the saddle-point equation.
\end{proof}

\begin{lemma}\label{lem:Hder}
$H(0)=0$, $H'(0)=\ii c_1$, $H^{(r)}(0)=\ii^rc_r$ for $r\ge2$; $|H'|\le c_1$ and $|H^{(r)}|\le c_r$ $(r\ge2)$ on $\RR$; and $\re H\le0$.
\end{lemma}

\begin{proof}
$H=\ii\varphi+\ell_{q,K}+\log(1-q)$, so $H'=\ii+\ii K\sum_lq^le^{\ii lK\varphi}$ and $|H'|\le1+Kq/(1-q)=c_1$; for $r\ge2$,
$H^{(r)}=\ell_{q,K}^{(r)}$. Finally $|1-qe^{\ii K\varphi}|^2-(1-q)^2=2q(1-\cos K\varphi)\ge0$, so $|G|\le g_0$ and $\re H=\log|G/g_0|\le0$.
\end{proof}

\begin{corollary}\label{cor:omega}
Let $\omega:=\psi+B\varphi^2/2$. Then $\omega(0)=\omega'(0)=\omega''(0)=0$ and $|\omega'''|\le\kappa_3$; hence $|\omega(\varphi)|\le\kappa_3|\varphi|^3/6$.
Moreover $|\im\psi(\varphi)|\le\kappa_3|\varphi|^3/6$ and $\re\psi(\varphi)\le-B\varphi^2/2+\kappa_4\varphi^4/24$.
\end{corollary}

\begin{proof}
The first statements follow from Lemma~\ref{lem:psider} and Taylor's theorem with integral remainder (applied to real and
imaginary parts). $\im\psi$ is odd and vanishes to second order at $0$ because $\psi'(0)=0$ and $\psi''(0)=-B$ is real.
$\re\psi$ is even, with $(\re\psi)''(0)=-B$, vanishing odd derivatives at $0$, and $|(\re\psi)''''|\le\kappa_4$.
\end{proof}

We shall also use two elementary facts repeatedly.
\begin{enumerate}[label=(F\arabic*)]
\item\label{F1} If $\re z\le0$ then $|e^z-1|\le|z|$ and $|e^z-1-z-z^2/2|\le|z|^3/6$; more generally
  $|e^z-1|\le|z|\max(1,e^{\re z})$. (From $e^z-1=z\int_0^1e^{sz}\dd s$ and $e^z-1-z-z^2/2=z^3\int_0^1\frac{(1-s)^2}{2}e^{sz}\dd s$.)
\item\label{F2} For $0\le A$ and $0\le t\le1$, $\log(1+At)\ge t\log(1+A)$ (concavity).
\end{enumerate}

%======================================================================
\section{The main term and the error bound in the window}\label{sec:main}
%======================================================================

Throughout this section $2\le k\le n-2$ and $\theta$ is the saddle point. We fix the central arc
\[
  |\varphi|\le\varphi_0:=0.4\,\theta .
\]

\subsection{The Gaussian model}
Put $g(\varphi):=e^{-B\varphi^2/2}$, and let
\[
\begin{gathered}
  \omega_3=-\ii\kappa_3\varphi^3/6,\qquad \omega_4=\kappa_4\varphi^4/24,\qquad \omega_5=\ii\kappa_5\varphi^5/120,\\
  H_1=\ii c_1\varphi,\qquad H_2=-c_2\varphi^2/2,\qquad H_3=-\ii c_3\varphi^3/6,\qquad H_4=c_4\varphi^4/24,
\end{gathered}
\]
be the Taylor terms of $\omega=\psi+B\varphi^2/2$ and of $H$ at $0$ (Lemmas~\ref{lem:psider} and~\ref{lem:Hder}). Define the
polynomials
\[
  Q:=1+\omega_3+\omega_4+\omega_5+\tfrac12\omega_3^2+\omega_3\omega_4+\tfrac16\omega_3^3,\qquad
  P_H:=1+H_1+H_2+H_3+H_4+\tfrac12H_1^2,
\]
and the model integrals
\[
  D_0:=\int_\RR g\,Q\dd\varphi,\qquad \tilde N_0:=\int_\RR g\,Q\,(P_H-1)\dd\varphi,\qquad M_{\mathrm{model}}:=1+\tilde N_0/D_0 .
\]
They are evaluated exactly with the Gaussian moments $\int_\RR\varphi^{2j}e^{-A\varphi^2/2}\dd\varphi=\sqrt{2\pi/A}\,(2j-1)!!\,A^{-j}$.
In particular $D_0=\sqrt{2\pi/B}\,N_0$ with $N_0$ as in~\eqref{eq:M2}, and $\tilde N_0/D_0$ is real. We work with the scaled variables
\begin{equation}\label{eq:AC}
  A_r:=\frac{\kappa_r}{B^{r/2}}=\frac{s_r\,\theta^{r/2-1}}{b^{r/2}},\qquad C_r:=\frac{c_r}{B^{r/2}}=\frac{f_r(\nu)\,\theta^{r/2}}{b^{r/2}}\ \ (r\ge2),\qquad
  C_1:=\frac{c_1}{B^{1/2}},
\end{equation}
and $w_0:=\varphi_0\sqrt B=0.4\sqrt{b/\theta}$; every quantity below is invariant under $\kappa_r\mapsto\lambda^r\kappa_r$,
$c_r\mapsto\lambda^rc_r$, $B\mapsto\lambda^2B$, so it may be computed with $B=1$ and $(\kappa_r,c_r)$ replaced by $(A_r,C_r)$.

\begin{lemma}\label{lem:gap}
$M_{\mathrm{model}}-M_2=\mathrm{Num}/N_0$, where, in the scaled variables,
\[
\begin{aligned}
\mathrm{Num}={}&-\tfrac{25}{48}C_4A_3^2+\tfrac16C_4A_4+(C_2+C_1^2)\Bigl(\tfrac{1}{32}A_4^2+\tfrac{25}{192}A_3^4-\tfrac{25}{192}A_3^2A_4\Bigr)\\
&+\tfrac54C_3A_3^3+\tfrac{7}{48}C_3A_5-\tfrac{25}{24}C_3A_3A_4 .
\end{aligned}
\]
\end{lemma}

\begin{proof}
A polynomial identity in the variables $A_3,\dots,A_5,C_1,\dots,C_4$, obtained by expanding $\tilde N_0$ and $D_0$ in Gaussian
moments. It is verified symbolically (Table~\ref{tab:certs}, item~S2).
\end{proof}

On the window, $f_r(\nu)=O(q)=O(\theta)$, so by~\eqref{eq:AC} $A_r=O(\theta^{r/2-1})$, $C_r=O(\theta^{1+r/2})$ $(r\ge2)$ and
$C_1=O(\theta^{3/2})$. Hence every monomial of $\mathrm{Num}$ is $O(\theta^4)$, whereas $M_2-1=O(\theta^2)$ is of the same order as the
step $\theta q$ of $L_2$ (Theorem~\ref{thm:B}); this is why a second-order main term is needed.

\subsection{The pointwise error bound}
Let $a_r:=\kappa_r/r!$ and $t=|\varphi|$, and set
\[
  \delta_0:=\frac{\kappa_4\varphi_0^2/12+\kappa_6\varphi_0^4/360}{B},\qquad B':=(1-\delta_0)B .
\]
Define the nonnegative polynomials in $t$
\[
\begin{aligned}
T_\omega(t)&:=a_6t^6+\bigl(\tfrac12a_4^2+a_3a_5\bigr)t^8+a_4a_5t^9+\tfrac12a_5^2t^{10}\\
  &\qquad+\tfrac12(a_4t^4+a_5t^5)(a_3t^3+a_4t^4+a_5t^5)^2+\tfrac1{24}(a_3t^3+a_4t^4+a_5t^5)^4,\\
T_H(t)&:=\tfrac{1}{120}c_5t^5+\bigl(\tfrac12c_1c_2+\tfrac16c_1^3\bigr)t^3,
\end{aligned}
\]
and let $Q_{\mathrm{abs}}(t)$ and $P_{\mathrm{abs}}(t)$ be the polynomials obtained from $Q$ and $P_H-1$ by replacing each coefficient by
its modulus.

\begin{lemma}[Central arc]\label{lem:central}
Suppose $\delta_0<1$. For $|\varphi|\le\varphi_0$,
\[
  |e^{\omega}-Q|\le T_\omega(t)\,e^{\delta_0Bt^2/2},\qquad |e^H-P_H|\le T_H(t),\qquad |e^H-1|\le c_1t,
\]
and consequently
\[
\begin{aligned}
  \Bigl|\int_{|\varphi|\le\varphi_0}\bigl[\chi(e^H-1)-gQ(P_H-1)\bigr]\Bigr|&\le\int_\RR e^{-B't^2/2}c_1t\,T_\omega(t)\dd\varphi+\int_\RR g\,Q_{\mathrm{abs}}T_H\dd\varphi,\\
  \Bigl|\int_{|\varphi|\le\varphi_0}(\chi-gQ)\Bigr|&\le\int_\RR e^{-B't^2/2}\,T_\omega(t)\dd\varphi .
\end{aligned}
\]
\end{lemma}

\begin{proof}
By Taylor's theorem with $|\psi^{(6)}|\le\kappa_6$ and since the odd-order terms are purely imaginary,
$\re\omega\le\kappa_4t^4/24+\kappa_6t^6/720\le\delta_0Bt^2/2$ for $t\le\varphi_0$. Let $\hat\omega=\omega_3+\omega_4+\omega_5$; then
$|\omega-\hat\omega|\le a_6t^6$ and $\re\hat\omega=\kappa_4t^4/24\le\delta_0Bt^2/2$. We split
\[
  e^\omega-Q=(e^\omega-e^{\hat\omega})+\Bigl(e^{\hat\omega}-\sum_{j\le3}\tfrac{\hat\omega^j}{j!}\Bigr)+\Bigl(\sum_{j\le3}\tfrac{\hat\omega^j}{j!}-Q\Bigr).
\]
The first term is at most $|\omega-\hat\omega|e^{\max(\re\omega,\re\hat\omega)}$, the second is at most
$\frac{|\hat\omega|^4}{24}\max(1,e^{\re\hat\omega})$ (Taylor's theorem with integral remainder), and the third is the polynomial
$\frac12(\omega_4^2+\omega_5^2)+\omega_3\omega_5+\omega_4\omega_5+\frac16(\hat\omega^3-\omega_3^3)$, whose modulus is dominated termwise by
the remaining terms of $T_\omega$. For $T_H$, write
$e^H-P_H=(e^H-1-H-\frac12H^2)+(H-\sum_{r\le4}H_r)+\frac12(H-H_1)(H+H_1)$. By Lemma~\ref{lem:Hder}, $|H|\le c_1t$, so
\ref{F1} bounds the first piece by $c_1^3t^3/6$; Taylor's theorem with $|H^{(5)}|\le c_5$ bounds the second by $c_5t^5/120$; and
$|H-H_1|\le c_2t^2/2$, $|H+H_1|\le2c_1t$ bound the third by $c_1c_2t^3/2$. The bound $|e^H-1|\le c_1t$ is~\ref{F1}. The integral
inequalities follow from the identity
$\chi(e^H-1)-gQ(P_H-1)=g\bigl[(e^\omega-Q)(e^H-1)+Q(e^H-P_H)\bigr]$ and from $g\,e^{\delta_0Bt^2/2}=e^{-B't^2/2}$.
\end{proof}

The \emph{model tails} are bounded by $e^{-B\varphi^2/2}\le e^{-B\varphi_0^2/4}e^{-B\varphi^2/4}$ for $|\varphi|\ge\varphi_0$:
\begin{equation}\label{eq:tails}
\begin{aligned}
  \int_{|\varphi|>\varphi_0}g\,|Q(P_H-1)|&\le e^{-B\varphi_0^2/4}\int_\RR e^{-B\varphi^2/4}Q_{\mathrm{abs}}P_{\mathrm{abs}},\\
  \int_{|\varphi|>\varphi_0}g\,|Q|&\le e^{-B\varphi_0^2/4}\int_\RR e^{-B\varphi^2/4}Q_{\mathrm{abs}} .
\end{aligned}
\end{equation}
On the \emph{minor arcs} $\varphi_0\le|\varphi|\le\pi$, since $|e^H-1|\le2$ by $\re H\le0$ and the arcs have total length at most $2\pi$,
\begin{equation}\label{eq:minor}
  \Bigl|\int_{\varphi_0\le|\varphi|\le\pi}\chi(e^H-1)\Bigr|\le4\pi M_\chi,\qquad \Bigl|\int_{\varphi_0\le|\varphi|\le\pi}\chi\Bigr|\le2\pi M_\chi,
\end{equation}
where $M_\chi:=\sup_{\varphi_0\le|\varphi|\le\pi}|\chi(\varphi)|$.

\begin{proposition}[Pointwise error bound]\label{prop:eps}
Let $E_N$ (resp.\ $E_D$) be the sum of the numerator (resp.\ denominator) bounds of Lemma~\ref{lem:central},~\eqref{eq:tails}
and~\eqref{eq:minor}, all divided by $\sqrt{2\pi/B}$ (recall $D_0/\sqrt{2\pi/B}=N_0$). Suppose $\delta_0<1$, $N_0-E_D>0$ and
\[
  \Delta:=\frac{E_N+|\tilde N_0/D_0|\,E_D}{N_0-E_D}+\frac{|\mathrm{Num}|}{N_0}<M_2 .
\]
Then
\begin{equation}\label{eq:eps}
  |L_k-L_{2,k}|\le\varepsilon_k:=-\log\bigl(1-\Delta/M_2\bigr).
\end{equation}
\end{proposition}

\begin{proof}
Write $\tilde N=\tilde N_0+e_N$ and $D=D_0+e_D$. The three groups of estimates give $|e_N|\le\sqrt{2\pi/B}\,E_N$ and
$|e_D|\le\sqrt{2\pi/B}\,E_D$. Since
$\tilde N/D-\tilde N_0/D_0=(e_N-(\tilde N_0/D_0)e_D)/(D_0+e_D)$, Proposition~\ref{prop:rep} gives
$|R_k/g_0-M_{\mathrm{model}}|\le\Delta-|\mathrm{Num}|/N_0$, and Lemma~\ref{lem:gap} gives $|R_k/(g_0M_2)-1|\le\Delta/M_2<1$. Finally
$|\log(1+y)|\le-\log(1-|y|)$ for $|y|<1$, and $L_k-L_{2,k}=\log(R_k/(g_0M_2))$.
\end{proof}

All the integrals in $E_N$ and $E_D$ are integrals of nonnegative polynomials against Gaussians over $\RR$, which we evaluate exactly
via $\int_\RR|\varphi|^pe^{-A\varphi^2/2}\dd\varphi=(2/A)^{(p+1)/2}\Gamma(\frac{p+1}2)$. The full expression for $\varepsilon_k$, written
in the scaled variables~\eqref{eq:AC}, is recorded in Appendix~\ref{app:eps}. It is a fixed algebraic expression in
$(\theta,\nu',b,s_3,\dots,s_6)$ and in an upper bound for $M_\chi$; all of its algebra (the moments, the polynomial
identities above, and the scaling invariance) is checked symbolically (Table~\ref{tab:certs}, S2).

\subsection{The minor arcs}
The only non-local input is a bound for $M_\chi$. The following lemma gives it uniformly in $\theta$; its proof uses only the
partial sums of $\sin^2$, and no counting arguments.

\begin{lemma}[Minor arcs]\label{lem:minor}
Let $\theta_{\max}>0$, $u_{\mathrm{top}}>0$ and $c_0>0$, and suppose $\theta\le\theta_{\max}$ and $(k+1)\theta\ge u_{\mathrm{top}}$.
Let $c_a\ge c_0$ with $c_a\theta_{\max}/2\le\pi/2$, and put $L_\ast:=2\theta_{\max}+2/\bigl(c_a\sinc(c_a\theta_{\max}/2)\bigr)$.
\begin{enumerate}[label=(\alph*)]
\item If $\varphi/\theta\in[c_{\mathrm{lo}},c_{\mathrm{hi}}]\subset[c_0,c_a]$, then for every $T\in(0,\pi/2]$, with $\sigma=\sin T/T$,
\[
  \theta S(\varphi)\ \ge\ \int_{\theta_{\max}}^{\min(u_{\mathrm{top}},\,2T/c_{\mathrm{hi}})}\log\bigl(1+\sigma^2c_{\mathrm{lo}}^2f_2(u)\bigr)\dd u .
\]
\item If $c_a\theta\le\varphi\le\pi$, then $\theta S(\varphi)\ge\int_{L_\ast}^{u_{\mathrm{top}}}\log\coth(u/2)\dd u$ (to be read as $0$ if $L_\ast\ge u_{\mathrm{top}}$).
\end{enumerate}
Consequently $\log|\chi(\varphi)|\le-\Psi/(2\theta)$ for $c_0\theta\le|\varphi|\le\pi$, where $\Psi$ is the minimum of the right-hand
sides over a cover of $[c_0,c_a]$ by cells and of (b).
\end{lemma}

\begin{proof}
By~\eqref{eq:modchi} and $\chi(-\varphi)=\overline{\chi(\varphi)}$ we may take $\varphi>0$; recall that every summand of $S$ is $\ge0$.

(a) Put $c=\varphi/\theta$. For $i\le2T/\varphi$ we have $i\varphi/2\le T$, so $\sin(i\varphi/2)\ge\sigma\,i\varphi/2$, because $\sin y/y$ is
decreasing on $(0,\pi)$. Hence the $i$-th summand is at least
$\log\bigl(1+\sigma^2c^2(i\theta/2)^2/\sinh^2(i\theta/2)\bigr)\ge\lambda(i\theta)$ with $\lambda(u):=\log(1+\sigma^2c_{\mathrm{lo}}^2f_2(u))$,
a decreasing function. Keep only $i\le I:=\min(k,\lfloor2T/\varphi\rfloor)$. Then
$\theta S\ge\theta\sum_{i\le I}\lambda(i\theta)\ge\int_\theta^{(I+1)\theta}\lambda$. Since $(\lfloor2T/\varphi\rfloor+1)\theta>2T/c$, we have
$(I+1)\theta\ge\min((k+1)\theta,2T/c)\ge\min(u_{\mathrm{top}},2T/c_{\mathrm{hi}})$, and $\theta\le\theta_{\max}$, $\lambda\ge0$.

(b) Put $d=\varphi/2\in(0,\pi/2]$ and $w_i:=2\log\coth(i\theta/2)=\log(1+1/\sinh^2(i\theta/2))$, which is decreasing in $i$.
By~\ref{F2} with $t=\sin^2(id)$, the $i$-th summand is at least $\sin^2(id)\,w_i$. Let $G_I:=\sum_{i\le I}\sin^2(id)$. The Dirichlet
identity
\begin{equation}\label{eq:dirichlet}
  G_I=\frac I2-\frac{\sin(Id)\cos((I+1)d)}{2\sin d}
\end{equation}
gives $G_I\ge I/2-1/(2\sin d)$, and since also $G_I\ge0$ we get $G_I\ge\frac12(I-N_d)_+$ with $N_d:=\lceil1/\sin d\rceil$.
By Abel summation, with $w_{k+1}:=0$ and $w_I-w_{I+1}\ge0$,
\[
  S\ \ge\ \sum_{i\le k}\sin^2(id)w_i=\sum_{I\le k}G_I(w_I-w_{I+1})\ \ge\ \tfrac12\sum_{I\le k}(I-N_d)_+(w_I-w_{I+1})=\tfrac12\sum_{N_d<i\le k}w_i .
\]
As $w$ is decreasing, $w_i\ge\theta^{-1}\int_{i\theta}^{(i+1)\theta}2\log\coth(u/2)\dd u$, so
$\theta S\ge\int_{(N_d+1)\theta}^{(k+1)\theta}\log\coth(u/2)\dd u$. Finally $(N_d+1)\theta\le2\theta+\theta/\sin d$, and since $\sin$ is
increasing on $[0,\pi/2]$ and $\sinc$ is decreasing, $\sin d\ge\sin(c_a\theta/2)\ge(c_a\theta/2)\sinc(c_a\theta_{\max}/2)$. Hence
$(N_d+1)\theta\le L_\ast$ and $(k+1)\theta\ge u_{\mathrm{top}}$.
\end{proof}

\begin{corollary}\label{cor:cstar}
If $\theta\le0.0042$ and $(k+1)\theta\ge3$, then $|\chi(\varphi)|\le e^{-c_\ast/\theta}$ for $0.4\theta\le|\varphi|\le\pi$, with
$c_\ast=0.15973$.
\end{corollary}

\begin{proof}
Apply Lemma~\ref{lem:minor} with $\theta_{\max}=0.0042$, $u_{\mathrm{top}}=3$, $c_0=0.4$ and $c_a=0.4\cdot150^{96/400}=1.33144\ldots$;
then $L_\ast\le1.5105$. All integrals are right Riemann sums of decreasing integrands, hence rigorous lower bounds, evaluated in
ball arithmetic: (a) gives $\theta S\ge0.31946$ on the $96$ geometric cells of ratio $150^{1/400}$ covering $[0.4,c_a]$ (best of
$41$ values of $T$ per cell), and (b) gives
$\theta S\ge0.34433$. So $\Psi\ge0.31946$ and $c_\ast=\Psi/2\ge0.15973$ (Table~\ref{tab:certs}, C3).
\end{proof}

\subsection{Statement of Theorem A}
\begin{theorem}[Theorem A]\label{thm:A}
Let $n\ge10^5$ and $k\in W=(\mu-2\beta,\mu+2\beta)$. Then
\[
  |L_k-L_{2,k}|\le\varepsilon_k\le\Gamma_{\max}\,\theta q,\qquad \Gamma_{\max}=0.4282 .
\]
More precisely $\varepsilon_k\le0.42024\,\theta q$ when $\theta\ge0.002$, and $\varepsilon_k\le0.17605\,\theta q$ when $\theta\le0.002$.
\end{theorem}

The first inequality is Proposition~\ref{prop:eps}. The second is proved in Section~\ref{sec:proofA} after the ingredient bounds are in
place.

%======================================================================
\section{Ingredient bounds}\label{sec:ingredients}
%======================================================================

\subsection{The window}
Write $n=\zeta_2\beta^2$ and $\beta_A:=\sqrt{6\cdot10^5}/\pi=246.56\ldots$, $\ell_A=\log\beta_A$.

\begin{lemma}[Lemma W]\label{lem:W}
Let $n\ge10^5$, $k\in W$, and put $\delta:=5\ell/\beta$, $\epsilon_\beta:=(\ell+2+1/\beta)/(\zeta_2\beta)$ and $r:=(1-\epsilon_\beta)^{-1/2}$. Then
\[
  \frac{1-\delta}{\beta}<\theta\le\frac r\beta,
\]
and consequently
\[
  \theta\le\theta_A:=0.0040939,\qquad \nu'\in[-2.5102,\ 2.0840],\qquad \nu\ge3.1158 .
\]
\end{lemma}

\begin{proof}
\emph{Two-sided bound for $t^2S_k(t)$.} We have $t^2S_k(t)=t\sum_{i\le k}f_1(it)$ with $f_1$ decreasing and $0<f_1\le1$. Comparing with
integrals, $F_1((k+1)t)-F_1(t)\le t^2S_k(t)\le F_1(kt)\le\zeta_2$, and $F_1(t)\le t$. Write $F_1(v)=\zeta_2-T(v)$ with
$T(v)=\int_v^\infty f_1=\sum_l e^{-lv}(v/l+1/l^2)\le(v+1)e^{-v}/(1-e^{-v})$.

\emph{Upper bound.} Since $k<\mu+2\beta$, $m>n-\beta(\ell+2)$, and $\theta^2m\le\zeta_2$ gives
$\theta^2\le\zeta_2/m\le r^2/\beta^2$.

\emph{Lower bound.} Let $\theta_-:=(1-\delta)/\beta$. Since $S_k$ is decreasing it suffices to show $S_k(\theta_-)>m$. Now
$\theta_-^2m\le\theta_-^2n=(1-\delta)^2\zeta_2$ and $\theta_-^2S_k(\theta_-)\ge\zeta_2-T(K\theta_-)-\theta_-$, so it suffices that
$(2\delta-\delta^2)\zeta_2>T(K\theta_-)+\theta_-$. As $K>\beta(\ell-2)$, $K\theta_-\ge(\ell-2)(1-\delta)=:v_-$, and the bound for $T$
turns the sufficient condition into $G>0$, where (after multiplication by $\beta/\ell$)
\[
  G:=2\cdot5\,\zeta_2-25\,\zeta_2\,\ell/\beta-e^2\beta^{\delta}/(1-e^{-v_-})-1/\ell .
\]

\emph{Parameter ranges.} At fixed $K$, $\nu'=K\theta+\log\theta$ is increasing in $\theta$. With $K<\mu+2\beta+1$ and $\theta\le r/\beta$
this gives $\nu'\le(\ell+2+1/\beta)r-\ell+\log r$; with $K>\mu-2\beta$ and $\theta>(1-\delta)/\beta$ it gives
$\nu'\ge-2-\delta(\ell-2)+\log(1-\delta)$ and $\nu\ge(\ell-2)(1-\delta)$.

\emph{Certification.} Each of the resulting expressions, and $G$, is a function of $x=1/\beta$, $\ell x$, $\ell^2x$ and $1/\ell$. An
adaptive ball-arithmetic bisection over $\ell\in[\ell_A,60]$ ($5486$ boxes), together with one tail box for $\ell\ge60$ (where
$x,\ell x,\ell^2x\le3600e^{-60}$ since $\ell^ae^{-\ell}$ is decreasing for $\ell>a$), certifies $G\ge1.0092$ and the stated
constants for every $\beta\ge\beta_A$ (Table~\ref{tab:certs}, C1). No monotonicity in $\beta$ is used.
\end{proof}

The same certification gives the constants used for the other regions:
\begin{equation}\label{eq:CWconst}
  \nu'\le\nu_{CL}:=-1.95361\ \text{ for }k\le\mu-2\beta,\qquad \nu'\ge\nu_{CR}:=1.04305\ \text{ for }k\ge\mu+2\beta,
\end{equation}
and the edge values $\nu'\le-1.906674$ at the left end of $W$ and $\nu'\ge1.04305$ at the right end. (For~\eqref{eq:CWconst} one uses
that $\theta_k$ is increasing in $k$, since $S_k(t)$ increases with $k$ while $m=n-k$ decreases.)

\subsection{Riemann sums}
\begin{lemma}\label{lem:riemann}
Let $f:[0,\infty)\to\RR$ have $\TV(f)<\infty$. Then for $\theta>0$ and $k\ge1$,
$\bigl|\theta\sum_{i\le k}f(i\theta)-\int_0^{k\theta}f\bigr|\le\theta\Var_{[0,k\theta]}f\le\theta\TV(f)$.
\end{lemma}

\begin{proof}
$\theta f(i\theta)-\int_{(i-1)\theta}^{i\theta}f=\int_{(i-1)\theta}^{i\theta}(f(i\theta)-f(u))\dd u$, and $|f(i\theta)-f(u)|\le\Var_{[(i-1)\theta,i\theta]}f$.
Summing over $i$ and using additivity of the variation gives the claim.
\end{proof}

\begin{lemma}\label{lem:TV}
The total variations of $f_1,\dots,f_6$ on $[0,\infty)$ satisfy
\[
  \TV(f_1),\dots,\TV(f_6)\ \le\ 1.000001,\ \ 1.000001,\ \ 2.000001,\ \ 6.086765,\ \ 25.691943,\ \ 138.425 .
\]
\end{lemma}

\begin{proof}
On $[60,\infty)$ each $f_r$ is decreasing to $0$ (each term $l^{r-1}u^re^{-lu}$ decreases for $u>r$), so the variation there is $f_r(60)$.
On $[0,60]$ we use $6000$ cells. On a cell where ball arithmetic certifies that $f_r'$ has constant sign, $f_r$ is monotone and the
variation is $|f_r(b)-f_r(a)|$. Otherwise the cell is bisected (to depth $14$), and on the final cells the variation is bounded by
$(b-a)\sup|f_r'|\ge\int_a^b|f_r'|$. For $u\le2.5$, $f_r$ and $f_r'$ are evaluated from their Bernoulli-number expansions with a rigorous
tail bound ($|B_j|/j!\le4/(2\pi)^j$); for $u>2.5$, from the polylogarithm and the identity $uf_r'=rf_r-f_{r+1}$
(Table~\ref{tab:certs}, C2).
\end{proof}

\begin{corollary}\label{cor:sr}
$s_r\in F_r(\nu-\theta)\pm\theta\,\TV(f_r)$ and $0\le s_r\le r!\zeta_2+\theta\TV(f_r)$. In particular
$b\ge F_2(\nu-\theta)-\theta\TV(f_2)$.
\end{corollary}

\begin{proof}
By~\eqref{eq:scaled}, $s_r=\theta\sum_{i\le k}f_r(i\theta)$; apply Lemmas~\ref{lem:riemann} and~\ref{lem:Fr} with $k\theta=\nu-\theta$.
\end{proof}

Note that $f_r(\nu)=c_r\theta^r$ and $c_1=1+\nu e^{-\nu'}/(1-q)$ are \emph{exact} functions of $(\theta,\nu')$. Thus every quantity in
Proposition~\ref{prop:eps} is enclosed, given $(\theta,\nu')$, by the enclosures of Corollary~\ref{cor:sr}, and $M_\chi$ by
Corollary~\ref{cor:cstar}.

\subsection{Proof of Theorem A}\label{sec:proofA}
By Lemma~\ref{lem:W}, for $k\in W$ we have $\theta\le\theta_A$, $\nu'\in[-2.52,2.10]$ and $\nu\ge3.1158$; in particular
Corollary~\ref{cor:cstar} applies. Since $\theta q=\theta^2e^{-\nu'}$, it suffices to bound $\Gamma:=\varepsilon_k/(\theta^2e^{-\nu'})$.

\emph{Region I: $0.002\le\theta\le\theta_A$.} We cover $\theta$ by geometric cells of ratio $1.01$ and $\nu'$ by cells of width $0.01$
(clipped to $\nu\ge3.1158$). On each box we evaluate $\varepsilon$ of Proposition~\ref{prop:eps} in ball arithmetic, with $s_r$ ranging over
the enclosure of Corollary~\ref{cor:sr} for every $(\theta,\nu')$ in the box, $f_r(\nu)$ and $c_1$ exact, and
$M_\chi\le e^{-0.15926/\theta}$ (a rounded-down value, valid since $0.15926\le c_\ast$); the code asserts $\delta_0<1$, $N_0-E_D>0$ and $\Delta<M_2$ on every box, and divides by
$\min\theta^2e^{-\nu'}$ over the box. The maximum is $\Gamma\le0.42024$, attained near $\nu'=2.1$ (Table~\ref{tab:certs}, A1).

\emph{Region II: $\theta\le0.002$.} Here no box cover is possible, and we use a monotone majorant in $t=\sqrt\theta$. After dividing by
$\theta q=t^2q$, every term of $\varepsilon$ is bounded by a nonnegative combination of monomials $t^P\nu^je^{-\beta_q\nu'}$ with
coefficients that are bounded uniformly through $s_r\le r!\zeta_2+0.002\TV(f_r)$, $b\ge F_2(\nu_{\min}-0.002)-0.002\TV(f_2)$ and
$e^{-\nu}\le e^{-\nu_{\min}}$. At fixed $\nu'$ we have $\nu=\nu'+2\log(1/t)$ and
\[
  \frac{d}{dt}\bigl[t^P\nu^j\bigr]=t^{P-1}\nu^{j-1}(P\nu-2j),
\]
which is nonnegative when $P>0$ and $\nu\ge2j/P$ (and trivially when $j<0$). The program asserts, for every
monomial, either $P=j=0$ (a constant term) or $P>0$ and $\nu_{\min}>2j/P$, and evaluates the majorant at $\theta=0.002$. The Gaussian-tail and minor-arc factors $e^{-0.04b/\theta}\theta^{-2}$ and
$e^{-c_\ast/\theta}\theta^{-3.5}$ are increasing on the relevant range because $\theta<0.02b$ and $\theta<c_\ast/3.5$, which is asserted. The
result is $\Gamma\le0.17605$ for all $\theta\le0.002$ and $\nu'\in[-2.52,2.10]$ (Table~\ref{tab:certs}, A2). The details of the majorant
are in Appendix~\ref{app:regionII}. \qed

%======================================================================
\section{The step of the main term}\label{sec:B}
%======================================================================

In this section $k$ and $k+1$ both lie in $W$, and primes refer to the index $k+1$: $\theta'=\theta_{k+1}$ is the saddle point for
$(m-1,k+1)$, $q'=e^{-(K+1)\theta'}$, $\nu_{k+1}=(K+1)\theta'$, and so on.

\begin{theorem}[Theorem B]\label{thm:B}
Let $n\ge10^5$ and $k,k+1\in W$. Then
\[
  L_{2,k}-L_{2,k+1}\ \ge\ \Sigma_{\min}\,\theta_kq_k,\qquad \Sigma_{\min}=0.91918,
\]
and $\theta_{k+1}q_{k+1}\le1.01\,\theta_kq_k$. In particular $L_2$ is strictly decreasing on $W$ and
\[
  L_{2,k}-L_{2,k+1}\ \ge\ 0.9100\,\max(\theta_kq_k,\theta_{k+1}q_{k+1})\ >\ 2\Gamma_{\max}\max(\theta_kq_k,\theta_{k+1}q_{k+1}) .
\]
\end{theorem}

\subsection{Decomposition}
\begin{equation}\label{eq:stepdec}
  L_{2,k}-L_{2,k+1}=(\theta'-\theta)+\log\frac{1-q'}{1-q}-\bigl(\log M_2'-\log M_2\bigr).
\end{equation}
We shall see that $\theta'>\theta$, so the first term is positive and $q'=e^{-(K+1)\theta'}<qe^{-\theta}$. Hence
\[
  \log\frac{1-q'}{1-q}\ \ge\ T_2:=\log\Bigl(1+\frac{q(1-e^{-\theta})}{1-q}\Bigr),
\]
which is $\theta q(1+O(\theta+q))$, and the step is at least $T_2-|\log M_2'-\log M_2|$. The work is to bound $\log M_2'-\log M_2$, for which we
need to control how $(\theta,\nu,b,s_3,s_4,s_5)$ moves from $k$ to $k+1$.

\subsection{Localising the next saddle point}
Let $B_k(t):=-S_k'(t)=t^{-3}b(k,t)$ with $b(k,t):=t\sum_{i\le k}f_2(it)$. Note $S_{k+1}(t)-S_k(t)=K/(e^{Kt}-1)$, which equals $c_1-1$
at $t=\theta$. Fix $\tau=0.01$ and let $b_{\min}:=F_2(\nu-\theta)-(1+\tau)\theta\TV(f_2)$. By Corollary~\ref{cor:sr} and monotonicity of
$F_2$, $b(k,t)\ge b_{\min}$ for $t\in[\theta,(1+\tau)\theta]$.

\begin{lemma}\label{lem:B1}
\begin{enumerate}[label=(\roman*)]
\item If $\tau b_{\min}>(1+\tau)^3c_1\theta^2$, then $\theta<\theta'\le(1+\tau)\theta$.
\item Under (i), $d\theta:=\theta'-\theta\le(1+\tau)^3c_1\theta^3/b_{\min}$.
\item $d\nu:=\nu_{k+1}-\nu=\theta'+K\,d\theta\in\bigl(0,(1+\tau)\theta+(\nu/\theta)\,d\theta\bigr]$.
\item For $r=2,\dots,5$, with $\bar s_r:=r!\zeta_2+(1+\tau)\theta\TV(f_r)$,
\[
  |s_r(k+1,\theta')-s_r(k,\theta)|\le d\theta\,\frac{(1+r)\bar s_r+\bar s_{r+1}}{\theta}+(1+\tau)\theta\sup_{u\in[\nu,\nu+d\nu]}f_r(u).
\]
\end{enumerate}
\end{lemma}

\begin{proof}
(i) $S_{k+1}(\theta)=m+c_1-1>m-1=S_{k+1}(\theta')$, so $\theta'>\theta$. On $[\theta,(1+\tau)\theta]$,
$-S_{k+1}'\ge B_k(t)\ge b_{\min}/((1+\tau)\theta)^3$, so $S_{k+1}((1+\tau)\theta)\le m+c_1-1-\tau b_{\min}/((1+\tau)^3\theta^2)<m-1$.
(ii) By the mean value theorem $c_1=S_{k+1}(\theta)-S_{k+1}(\theta')=B_{k+1}(\xi)\,d\theta\ge b_{\min}\,d\theta/((1+\tau)\theta)^3$.
(iii) $\nu_{k+1}-\nu=(K+1)\theta'-K\theta$ and $K=\nu/\theta$.
(iv) $s_r(k+1,\theta')=s_r(k,\theta')+\theta'f_r(K\theta')$ with $K\theta'\in[\nu,\nu+d\nu]$. From $uf_r'=rf_r-f_{r+1}$ we get
$\partial_ts_r(k,t)=((1+r)s_r(k,t)-s_{r+1}(k,t))/t$, and $0\le s_r(k,t)\le\bar s_r$ for $t\le(1+\tau)\theta$.
\end{proof}

\subsection{Proof of Theorem B}
\emph{Region I ($0.002\le\theta\le\theta_A$).} We cover $(\theta,\nu')$ by $8526$ boxes ($\theta$-ratio $1.02$, $\nu'$-width $0.02$, clipped to
$\nu\ge3.1158$). On each box the program asserts the hypothesis of Lemma~\ref{lem:B1}(i) and encloses $d\theta$, $d\nu$ and the
increments of $s_r$. It then forms a product box $\mathcal H\subset\RR^6$ in the variables $(\theta,\nu,b,s_3,s_4,s_5)$:
$\theta\in[\theta,(1+\tau)\theta]$, $\nu\in[\nu,\nu+d\nu]$ and
$s_r\in[F_r(\nu-\theta)-(1+\tau)\theta\TV_r,\ F_r(\nu+d\nu)+(1+\tau)\theta\TV_r]$. Both the point for $k$ and the point for $k+1$ lie in
$\mathcal H$ (Corollary~\ref{cor:sr} and Lemma~\ref{lem:B1}), and $\mathcal H$ is convex. The program checks $M_2>0$ on $\mathcal H$ and
encloses $\nabla M_2$ on $\mathcal H$ by forward-mode automatic differentiation in ball arithmetic, so by the mean value theorem along the
segment
\[
  |\log M_2'-\log M_2|\le\sum_j\sup_{\mathcal H}\Bigl|\frac{\partial_jM_2}{M_2}\Bigr|\,|v_j'-v_j|,
\]
where $v=(\theta,\nu,b,s_3,s_4,s_5)$ and $v'$ is the same vector at $k+1$, with $|v_j'-v_j|$ bounded by Lemma~\ref{lem:B1}. Together with $T_2/(\theta q)$ evaluated on the box this gives
$\Sigma:=(\text{step})/(\theta q)\ge0.91918$ on every box (Table~\ref{tab:certs}, B1).

\emph{Region II ($\theta\le0.002$).} The same graded-monomial method as in Section~\ref{sec:proofA} applies. At both points every
quantity is majorised by nonnegative combinations of $t^pq^{b_q}\nu^j$, using $\theta'^e\le(1+\tau)^e\theta^e$ ($e>0$), $q'\le q$,
$\nu_{k+1}\le\rho\nu$ with an explicit $\rho$, $f_r(u)\le(\rho\nu)^rqP_{r-1}(y)/(1-y)^r$ ($y=e^{-\nu_{\min}}$, $P_{r-1}$ the Eulerian
polynomial), and $|h'(u)|\le(u+1)e^{-u}/(1-e^{-u})^2$ for $c_1=1+h(\nu)/\theta$, $h(u)=u/(e^u-1)$. Products and quotients are
handled by $|x'y'-xy|\le|x'-x|\sup|y|+\sup|x|\,|y'-y|$ and $|1/N'-1/N|\le|N'-N|/N_{\min}^2$ (primes denoting values at $k+1$), and the main term by
$T_2/(\theta q)\ge(1-\theta/2)\bigl(1-\theta q/(2(1-q))\bigr)$. The result is $\Sigma\ge0.96039$ for all $\theta\le0.002$ and
$\nu'\in[-2.52,2.10]$, and also $|\log M_2|/\theta\le0.14215$ (Table~\ref{tab:certs}, B2).

Finally $\theta'q'\le1.01\,\theta q$ because $\theta'\le1.01\theta$ and $q'<q$, and $0.91918/1.01>0.9100>0.8564=2\Gamma_{\max}$. \qed

\subsection{Consequences in the window}
Let $k^\ast$ be the largest $k\in W$ with $L_{2,k}>0$, and $k^\ast:=\min(W\cap\ZZ)-1$ if there is no such $k$ (so $L_{2,k}\le0$ for every $k\in W$ with $k>k^\ast$). Call $k\in W$
\emph{undetermined} if $|L_{2,k}|\le\varepsilon_k$. If $k$ is not undetermined then $\operatorname{sign}L_k=\operatorname{sign}L_{2,k}$
by Theorem~\ref{thm:A}.

\begin{corollary}\label{cor:onek}
Let $n\ge10^5$. At most one $k\in W$ is undetermined, and if there is one it belongs to $\{k^\ast,k^\ast+1\}$. Hence the signs of
$L_k$ for $k\in W$, read left to right, have the form $+\cdots+-\cdots-$ (either block may be empty).
\end{corollary}

\begin{proof}
For $k\le k^\ast-1$, Theorem~\ref{thm:B} gives $L_{2,k}\ge L_{2,k}-L_{2,k+1}>0.91\,\theta_kq_k>\Gamma_{\max}\theta_kq_k\ge\varepsilon_k$.
For $k\ge k^\ast+2$, $L_{2,k}\le L_{2,k}-L_{2,k-1}\le-0.91\,\theta_kq_k$. So only $k^\ast$ and $k^\ast+1$ can be undetermined. If both
were, then
\[
  L_{2,k^\ast}-L_{2,k^\ast+1}\le\varepsilon_{k^\ast}+\varepsilon_{k^\ast+1}\le2\Gamma_{\max}\max(\theta q)<0.91\max(\theta q),
\]
contradicting Theorem~\ref{thm:B}. The determined signs are $+$ for $k\le k^\ast$ and $-$ for $k\ge k^\ast+1$, and a single undetermined
index adjacent to the change cannot create a second change.
\end{proof}

\begin{remark}[Location of the mode]
Using the edge values of $\nu'$ after Lemma~\ref{lem:W} one finds $L_2/\theta\ge5.588$ at the left end of $W$ and
$L_2/\theta\le-0.505$ at the right end (Table~\ref{tab:certs}, B3), so the sign change of $L_2$, and hence the mode of the row, lies
inside $W$. This is not needed for unimodality.
\end{remark}

%======================================================================
\section{The regions D, CL, CR and E}\label{sec:CLCR}
%======================================================================

\subsection{A zeroth-order estimate}
Outside the window a much cruder version of Proposition~\ref{prop:eps} suffices. Let $c_0>0$ and $\varphi_0=c_0\theta$.

\begin{lemma}[Lemma C0]\label{lem:C0}
Suppose $\re\psi(\varphi)\le-(1-\delta)B\varphi^2/2$ for $|\varphi|\le\varphi_0$, with $0\le\delta<1$. Put
\[
  T_1=\frac{c_1\kappa_3}{2B^2},\quad T_2=\frac{c_2+c_1^2}{2B},\quad A_3=\kappa_3B^{-3/2},\quad w_0=\varphi_0\sqrt B,\quad
  \mathcal M=\sqrt{\frac B{2\pi}}\int_{\varphi_0\le|\varphi|\le\pi}|\chi| .
\]
If the denominator below is positive, then
\[
  \Bigl|\frac{R_k}{g_0}-1\Bigr|\le\rho:=\frac{T_1(1-\delta)^{-5/2}+T_2(1-\delta)^{-3/2}+2\mathcal M}
  {1-\erfc(w_0/\sqrt2)-\sqrt{2/\pi}\,A_3/\bigl(3(1-\delta)^2\bigr)-\mathcal M}.
\]
Moreover $\mathcal M\le\sqrt{2\pi B}\,\sup_{\varphi_0\le|\varphi|\le\pi}|\chi|$.
\end{lemma}

\begin{proof}
We use~\eqref{eq:rep2} and measure all integrals in units of $\sqrt{2\pi/B}$.
\emph{Numerator, central arc.} Since $(e^H)''=(H''+H'^2)e^H$ and $\re H\le0$, Taylor's theorem gives $e^H-1=\ii c_1\varphi+E(\varphi)$ with
$|E|\le(c_2+c_1^2)\varphi^2/2$. By $\chi(-\varphi)=\overline{\chi(\varphi)}$, $\int_{|\varphi|\le\varphi_0}\chi\varphi=2\ii\int_0^{\varphi_0}\varphi\im\chi$,
and $|\im\chi|=e^{\re\psi}|\sin\im\psi|\le e^{-(1-\delta)B\varphi^2/2}\kappa_3|\varphi|^3/6$ by Corollary~\ref{cor:omega}. Integrating
$\varphi^4$ and $\varphi^2$ against $e^{-(1-\delta)B\varphi^2/2}$ over $\RR$ gives the terms $T_1(1-\delta)^{-5/2}$ and $T_2(1-\delta)^{-3/2}$.
\emph{Numerator, minor arcs.} $|e^H-1|\le2$ gives $2\mathcal M$.
\emph{Denominator.} $D$ is real and positive. On $|\varphi|\le\varphi_0$ write $\chi=e^{-B\varphi^2/2}e^\omega$; by
Corollary~\ref{cor:omega} and~\ref{F1}, $|e^\omega-1|\le|\omega|e^{\delta B\varphi^2/2}\le\frac{\kappa_3|\varphi|^3}6e^{\delta B\varphi^2/2}$, whose
weighted integral is $\sqrt{2/\pi}A_3/(3(1-\delta)^2)$. Also $\int_{|\varphi|\le\varphi_0}e^{-B\varphi^2/2}=\sqrt{2\pi/B}\,(1-\erfc(w_0/\sqrt2))$,
and the minor arcs contribute at most $\mathcal M$.
\end{proof}

\begin{lemma}[Admissible $\delta$]\label{lem:delta}
Both $\delta_T:=\kappa_4\varphi_0^2/(12B)$ and, when $c_0k\theta<2$,
$\delta_C:=1-\bigl(1-(c_0k\theta)^2/12\bigr)\log(1+c_0^2)/c_0^2$ satisfy the hypothesis of Lemma~\ref{lem:C0} (when they are $<1$).
\end{lemma}

\begin{proof}
For $\delta_T$ use Corollary~\ref{cor:omega}: $\re\psi\le-B\varphi^2/2+\kappa_4\varphi_0^2\varphi^2/24$. For $\delta_C$, by~\eqref{eq:modchi},
$-2\re\psi=\sum_i\log(1+\sin^2(i\varphi/2)/\sinh^2(i\theta/2))$. For $i\le k$ and $|\varphi|\le c_0\theta$ the angle $y=i|\varphi|/2$ is
$<1$, and $\sin^2y\ge(y-y^3/6)^2\ge y^2(1-y^2/3)\ge\lambda y^2$ with $\lambda:=1-(c_0k\theta)^2/12$. With
$a_i:=(i\varphi/2)^2/\sinh^2(i\theta/2)\le c_0^2$, two applications of~\ref{F2} give
$\log(1+\lambda a_i)\ge\lambda\log(1+a_i)\ge\lambda a_i\log(1+c_0^2)/c_0^2$, and $\sum_ia_i=B\varphi^2$.
\end{proof}

From $\log R_k=-\theta-\log(1-q)+\log(R_k/g_0)$, $q\le-\log(1-q)\le q/(1-q)$ and $1-q\le\nu$ we obtain, whenever $\rho<1$,
\begin{equation}\label{eq:C0cons}
  \log R_k\le-\theta+\frac q{1-q}+\rho,\qquad \log R_k\ge-\theta+q+\log(1-\rho),\qquad \log R_k\ge-\theta-\log\nu+\log(1-\rho).
\end{equation}

\subsection{A bound uniform in \texorpdfstring{$\theta$}{theta}}
\begin{lemma}[Lemma C1]\label{lem:C1}
Fix $0<v_1<v_2\le\infty$, $c_0>0$ and $\bar\theta>0$. Suppose $\Psi>5\bar\theta$ is such that $\theta S(\varphi)\ge\Psi$ for
$c_0\theta\le|\varphi|\le\pi$ whenever $\theta\le\bar\theta$ and $k\theta\in[v_1,v_2]$. Then there is an explicit number
$X=X(\bar\theta,v_1,v_2,c_0)$ such that $\rho\le\theta X$ for every saddle point $\theta\le\bar\theta$ with $k\theta\in[v_1,v_2]$.
\end{lemma}

\begin{proof}
Write $v=k\theta$. We bound the ingredients of Lemma~\ref{lem:C0}:
\begin{itemize}
\item $s_r\le\min\bigl(F_r(v_2)+\theta\TV(f_r),\ v\sup f_r\bigr)$ for $r=3,4$ (Corollary~\ref{cor:sr}, and each term $\theta f_r(i\theta)\le\theta\sup f_r$);
  with $\sup f_3\le2.0302009$, $\sup f_4\le6.1253939$;
\item $b\ge\max\bigl(F_2(v_1)-\theta,\ v_1f_2(v_2)\bigr)$ and $b\le F_2(v_2)$, by comparing the decreasing $f_2\le1$ with its integral;
\item $c_1\theta=\theta+f_1(\nu)\le\theta+f_1(v_1)$ and $c_2\theta^2=f_2(\nu)\le f_2(v_1)$, as $\nu\ge v\ge v_1$;
\item $\sup|\chi|\le e^{-\Psi/(2\theta)}$, hence $\mathcal M\le\sqrt{2\pi b}\,\theta^{-3/2}e^{-\Psi/(2\theta)}$.
\end{itemize}
In scaled form $T_1=\theta(c_1\theta)s_3/(2b^2)$, $T_2=\theta(c_2\theta^2+(c_1\theta)^2)/(2b)$, $A_3=s_3\theta^{1/2}b^{-3/2}$,
$w_0=c_0(b/\theta)^{1/2}$ and $\delta_T=c_0^2s_4/(12b)$. The quantity $\rho$ is increasing in $T_1,T_2,\mathcal M,\delta,A_3$ and
$\erfc(w_0/\sqrt2)$ (its numerator increases and its denominator decreases in each). After substituting the bounds above,
$T_1/\theta$, $T_2/\theta$, $A_3$ and $\delta$ are nondecreasing in $\theta$ and $w_0$ is nonincreasing; this is checked by computing
each derivative symbolically as a polynomial with nonnegative coefficients over a positive denominator (Table~\ref{tab:certs}, M).
Finally $\frac{d}{d\theta}\log(\mathcal M/\theta)=(\Psi-5\theta)/(2\theta^2)>0$ for $\theta<\Psi/5$. Hence $\rho/\theta$ is bounded by its
value with all ingredients taken at $\bar\theta$, which defines $X$. (The program takes the best of $c_0\in\{0.4,0.6,0.8,1,1.25,1.5,2\}$
and of $\delta_T,\delta_C$.)
\end{proof}

In this section $\Psi$ is provided by Lemma~\ref{lem:minor} with $u_{\mathrm{top}}=v_1$ (since $(k+1)\theta\ge k\theta\ge v_1$).

\subsection{The right tail CR}
\begin{proposition}\label{prop:CR}
Let $n\ge10^5$ and $\mu+2\beta\le k\le\lceil n/2\rceil$. Then $\log R_k\le-0.614\,\theta<0$.
\end{proposition}

\begin{proof}
Since $m\ge n-\lceil n/2\rceil\ge(n-1)/2$ and $\theta^2m\le\zeta_2$, $\theta\le\theta_{CR}:=\sqrt{2\zeta_2/(n-1)}\le0.0057358$. By~\eqref{eq:CWconst},
$\nu'\ge\nu_{CR}=1.04305$, so $v=\nu'-\log\theta-\theta\ge1.04305+\log(1/\theta_{CR})-\theta_{CR}\ge6.198$; both bounds improve with $n$.
Lemma~\ref{lem:minor} with $\theta_{\max}=\theta_{CR}$ and $u_{\mathrm{top}}=6.198$ gives $\Psi=0.33808>5\theta_{CR}$ (\script{ws\_mech\_cd}), and Lemma~\ref{lem:C1} on the single cell
$[6.198,\infty)$ (where only $F_r(\infty)=r!\zeta_2$ is used) gives $X\le0.0321$. By~\eqref{eq:C0cons}, with $q/\theta=e^{-\nu'}$ and $q\le e^{-v}$,
\[
  \log R_k\le\theta\Bigl(\frac{e^{-\nu'}}{1-q}+X-1\Bigr)\le\theta(0.3852-1).
\]
(Table~\ref{tab:certs}, R1.)
\end{proof}

\subsection{The left tail CL}
\begin{proposition}[CL-a]\label{prop:CLa}
For every $n$, if $47\le k\le n-2$ and $k\theta\le0.1$, then $R_k>1$.
\end{proposition}

\begin{proof}
Put $K_0=47$, $V_0=0.1$, $\theta_0=V_0/K_0$, so $\theta\le\theta_0$. Here $b\ge vf_2(V_0)$, $s_r\le v\sup f_r$, $c_1\theta\le1+\theta_0$ and
$f_2\le1$, which give $T_1\le(1+\theta_0)\sup f_3/(2K_0f_2(V_0)^2)$, $T_2\le(1+(1+\theta_0)^2)/(2K_0f_2(V_0))$,
$A_3\le\sup f_3\,K_0^{-1/2}f_2(V_0)^{-3/2}$ and $w_0\ge c_0(K_0f_2(V_0))^{1/2}$. We bound $\mathcal M\le\mathcal N_1+\mathcal N_2+\mathcal N_{34}$
according to $\varphi\in[c_0\theta,10\theta]$, $[10\theta,\pi/k]$ and $[\pi/k,\pi]$; since $10\theta=10v/k<\pi/k$ these cover the minor arcs.

\emph{$\mathcal N_1$.} Split $c=\varphi/\theta\in[c_0,10]$ into $60$ cells $[c_j,c_{j+1}]$. As $i\varphi/2\le c_{j+1}V_0/2=:T_j<\pi/2$,
$\sin(i\varphi/2)\ge\sigma_j\,i\varphi/2$ with $\sigma_j=\sin T_j/T_j$, and by~\ref{F2} each summand of $S$ is at least
$\lambda_jc^2f_2(i\theta)$ with $\lambda_j=\log(1+\sigma_j^2c_{j+1}^2)/c_{j+1}^2$. Since $\sum_ic^2f_2(i\theta)=B\varphi^2$, this gives
$|\chi|\le e^{-\lambda_jB\varphi^2/2}$, and
$2\sqrt{B/2\pi}\int_{c_j\theta}^\infty e^{-\lambda_jB\varphi^2/2}\dd\varphi=\erfc\bigl(c_j\theta\sqrt{\lambda_jB/2}\bigr)/\sqrt{\lambda_j}$ with
$\theta^2B=b/\theta\ge K_0f_2(V_0)$.

\emph{$\mathcal N_2$.} Here every angle $i\varphi/2$ is at most $\pi/2$, hence $\sin(i\varphi/2)\ge i\varphi/\pi$, and with
$c=\varphi/\theta$ each summand of $S$ is at least $\log\bigl(4c^2f_2(V_0)/\pi^2\bigr)$. Hence, using $\theta\sqrt B\le\sqrt k$,
\[
  |\chi|\le\Bigl(\frac{4f_2(V_0)}{\pi^2}\Bigr)^{-k/2}c^{-k},\qquad
  \mathcal N_2\le\sqrt{\frac2\pi}\,\sqrt k\,\Bigl(\frac{4f_2(V_0)}{\pi^2}\Bigr)^{-k/2}\frac{10^{1-k}}{k-1}.
\]

\emph{$\mathcal N_{34}$.} Put $d=\varphi/2\in[\pi/(2k),\pi/2]$. By~\ref{F2} each summand is at least $\sin^2(id)w_i$ with
$w_i=\log(1+1/\sinh^2(i\theta/2))\ge w_k$, and by~\eqref{eq:dirichlet} $\sum_{i\le k}\sin^2(id)\ge k/2-1/(2\sin d)$. Since
$\sin d\ge(\pi/2k)\sinc(\pi/(2K_0))$, this is at least $2gk$ with $g:=\frac12\bigl(\frac12-1/(\pi\sinc(\pi/2K_0))\bigr)>0$. Also
$w_k\ge-2\log\sinh(v/2)\ge-2\log(v\sigma_0/2)$ with $\sigma_0=\sinh(V_0/2)/(V_0/2)$. So $|\chi|\le(v\sigma_0/2)^{2gk}$, and with
$B\le k^3/v^2$, $\mathcal N_{34}\le\sqrt{2\pi}\,k^{3/2}v^{-1}(v\sigma_0/2)^{2gk}$, which is increasing in $v$ and decreasing in $k$ (both
asserted), so it is evaluated at $(K_0,V_0)$.

With $c_0=0.6$ and $\delta=\min(\delta_T,\delta_C)$ this gives $\mathcal N_1\le3.7\cdot10^{-4}$, $\mathcal N_2\le10^{-37}$,
$\mathcal N_{34}\le6.4\cdot10^{-8}$ and $\rho\le0.0673$. By the third inequality of~\eqref{eq:C0cons}, $\log R_k>0$ as soon as
$1-\rho>e^\theta\nu$, and $e^\theta\nu\le e^{\theta_0}V_0(1+1/K_0)\le0.1024$ (Table~\ref{tab:certs}, R2).
\end{proof}

\begin{proposition}[CL-b]\label{prop:CLb}
Let $n\ge10^5$ and $k_D<k\le\mu-2\beta$ with $k\theta\ge0.1$. Then $R_k>1$.
\end{proposition}

\begin{proof}
Here $k\ge79$, $\theta\le\theta_A$ and $\nu'\le\nu_{CL}=-1.95361$ by~\eqref{eq:CWconst}. Cover $v=k\theta\in[0.1,\infty)$ by the $32$ geometric
cells $[v_1,v_2]$ of ratio $1.15$ starting at $0.1$ and the final cell $[0.1\cdot1.15^{32},\infty)=[8.7565\ldots,\infty)$. On each cell $\theta\le\bar\theta:=\min(\theta_A,v_2/78)$; let
$X$ be the constant of Lemma~\ref{lem:C1}. By~\eqref{eq:C0cons}, $\log R_k\ge-\theta+q-\theta X/(1-\theta X)$, which is positive as soon as
$q/\theta>1+X/(1-\bar\theta X)$. For $q/\theta$ we use the larger of $e^{-\nu'}\ge e^{1.95361}$ and $e^{-(v_2+\bar\theta)}/\bar\theta$. Every
cell passes, with minimum ratio $5.82$ (Table~\ref{tab:certs}, R3).
\end{proof}

\subsection{Small \texorpdfstring{$k$}{k}: the region D}\label{sec:D}
\begin{lemma}\label{lem:D1}
For $1\le k\le n$, $\binom{n-1}{k-1}\le k!\,p(n,k)\le\binom{n-1+k(k-1)/2}{k-1}$.
\end{lemma}

\begin{proof}
Sorting a composition of $n$ into $k$ positive parts gives a partition into $k$ parts, and each partition arises from at most $k!$
compositions. Conversely $\lambda\mapsto(\lambda_i+k-i)_{i\le k}$ maps partitions of $n$ into $k$ parts injectively to partitions of
$n+k(k-1)/2$ into $k$ distinct parts, each of which arises from exactly $k!$ compositions.
\end{proof}

\begin{proposition}\label{prop:D}
Let $n\ge10^5$ and $1\le k\le k_D(n):=\lfloor1.7n^{1/3}\rfloor$. Then $p(n,k+1)>p(n,k)$.
\end{proposition}

\begin{proof}
Let $\Phi(k,n):=\log(n-k)-\log(k(k+1))-k(k-1)^2/(2(n-k+1))$. \emph{Step 1: $\Phi>0$ implies the claim.} By Lemma~\ref{lem:D1} it suffices
that $\binom{n-1}{k}/\binom{N-1}{k-1}>k+1$ with $N=n+D$, $D=k(k-1)/2$. The left side equals
$\frac{n-k}k\prod_{j=0}^{k-2}(1+D/(n-1-j))^{-1}\ge\frac{n-k}k\exp(-(k-1)D/(n-k+1))$, by $1+y\le e^y$ and $n-1-j\ge n-k+1$.

\emph{Step 2: monotonicity in $k$.} For real $1\le k<n$,
\[
  \partial_k\Phi=-\frac1{n-k}-\frac1k-\frac1{k+1}-\frac{(k-1)\bigl[(3k-1)(n-k+1)+k(k-1)\bigr]}{2(n-k+1)^2}<0 .
\]
Put $\gamma=1.7$ and $s=n^{1/3}\ge s_A:=10^{5/3}$; since $\gamma s<n-1$, $\Phi(k,n)\ge\Phi(\gamma s,s^3)$ for $k\le\gamma s$.

\emph{Step 3: reduction to one variable.} Using $k(k-1)^2/(n-k+1)\le k^3/(n-k)$, which follows from
$k^3(n-k+1)-k(k-1)^2(n-k)=k[(n-k)(2k-1)+k^2]\ge0$, we get $\Phi(\gamma s,s^3)\ge\Psi_D(s)$, where
\[
  \Psi_D(s):=\log\frac{s^2-\gamma}{\gamma(\gamma s+1)}-\frac{\gamma^3}{2(1-\gamma/s^2)},
\]
and, exactly,
\[
  \Psi_D'(s)=\frac{\gamma s^2+2s+\gamma^2}{(s^2-\gamma)(\gamma s+1)}+\frac{\gamma^4s}{(s^2-\gamma)^2}>0\quad(s>\sqrt\gamma).
\]
\emph{Step 4.} $\Psi_D(s_A)=0.30456\ldots>0$ (ball arithmetic), so $\Psi_D(s)>0$ for all $s\ge s_A$ (Table~\ref{tab:certs}, R4).
\end{proof}

\subsection{Large \texorpdfstring{$k$}{k}: the region E}
\begin{lemma}\label{lem:E}
If $2k\ge n$ then $p(n,k)=p(n-k)$. Consequently $p(n,k+1)\le p(n,k)$ for $\lceil n/2\rceil\le k\le n-1$, with equality only for $k=n-1$.
\end{lemma}

\begin{proof}
Subtracting $1$ from each part is a bijection from partitions of $n$ into exactly $k$ parts to partitions of $n-k$ into at most $k$ parts,
and the latter condition is vacuous when $n-k\le k$. For $k\ge\lceil n/2\rceil$ also $2(k+1)\ge n$, so $p(n,k+1)=p(j)$ and $p(n,k)=p(j+1)$
with $j=n-k-1$. Appending a part $1$ injects partitions of $j$ into those of $j+1$, and misses $(j+1)$ when $j\ge1$.
\end{proof}

%======================================================================
\section{Proof of the Main Theorem for \texorpdfstring{$n\ge10^5$}{n >= 10\^5}}\label{sec:assembly}
%======================================================================

Let $n\ge10^5$. The five ranges
\[
  D=[1,k_D],\ CL=(k_D,\mu-2\beta],\ W=(\mu-2\beta,\mu+2\beta),\ CR=[\mu+2\beta,\lceil n/2\rceil],\ E=[\lceil n/2\rceil,n-1]
\]
cover $[1,n-1]$, because $\mu+2\beta=\beta(\ell+2)\le n/2-1$: indeed $\beta(\ell+2)/(n/2)=(2/\zeta_2)(\ell+2)e^{-\ell}$ is decreasing for
$\ell>-1$ and is below $0.04$ at $n=10^5$. (If $k_D\ge\mu-2\beta$, the range $CL$ is empty and nothing changes.) Every $k$ in $CL\cup W\cup CR$
satisfies $79\le k\le\lceil n/2\rceil\le n-2$, so Proposition~\ref{prop:rep} applies there. The hypotheses of every lemma used hold for every
$n\ge10^5$; this bookkeeping is collected in Table~\ref{tab:thresholds} and checked by a separate program (Table~\ref{tab:certs}, T).

Now:
\begin{itemize}
\item $R_k>1$ for $k\in D$ (Proposition~\ref{prop:D});
\item $R_k>1$ for $k\in CL$: such $k$ satisfy $k\ge79\ge47$, so Proposition~\ref{prop:CLa} applies if $k\theta\le0.1$ and
  Proposition~\ref{prop:CLb} if $k\theta\ge0.1$;
\item on $W$ the signs of $\log R_k$ have the form $+\cdots+-\cdots-$ (Corollary~\ref{cor:onek});
\item $R_k<1$ for $k\in CR$ (Proposition~\ref{prop:CR}) and $R_k\le1$ for $k\in E$ (Lemma~\ref{lem:E}).
\end{itemize}
Concatenating, the signs of $p(n,k+1)-p(n,k)$, $1\le k\le n-1$, change at most once, from nonnegative to nonpositive. Hence the row is
unimodal. \qed

\begin{table}[ht]
\caption{Hypotheses of the lemmas and why they hold for all $n\ge10^5$.}\label{tab:thresholds}
\small
\begin{tabular}{@{}>{\raggedright\arraybackslash}p{0.27\textwidth}>{\raggedright\arraybackslash}p{0.33\textwidth}>{\raggedright\arraybackslash}p{0.33\textwidth}@{}}
\toprule
Lemma & Hypothesis & Holds because \\
\midrule
Cor.~\ref{cor:cstar} (minor arcs, window) & $\theta\le0.0042$, $(k+1)\theta\ge3$ & Lemma~\ref{lem:W}: $\theta\le0.0040939$, $\nu\ge3.1158$ \\
Thm.~\ref{thm:A}, Region I & $\theta\in[0.002,\theta_A]$, $\nu'\in[-2.52,2.10]$, $\nu\ge3.1158$ & Lemma~\ref{lem:W} \\
Thm.~\ref{thm:A}, Region II & $\theta\le0.002$, $\nu'\in[-2.52,2.10]$, $\theta<c_\ast/3.5$ & Lemma~\ref{lem:W}; $0.15926/3.5>0.002$ \\
Thm.~\ref{thm:B} & same ranges at $k$; hypothesis of Lemma~\ref{lem:B1}(i) & asserted on every box and cell \\
Cor.~\ref{cor:onek} & $\Sigma_{\min}/1.01>2\Gamma_{\max}$ & $0.91008>0.8564$ \\
Prop.~\ref{prop:CR} & $\theta\le\theta_{CR}$, $v\ge6.198$, $\Psi>5\theta_{CR}$ & $\theta_{CR}$ decreasing in $n$; $0.33808>0.0287$ \\
Prop.~\ref{prop:CLa} & $k\ge47$, $k\theta\le0.1$ & independent of $n$ \\
Prop.~\ref{prop:CLb} & $k\ge79$, $\theta\le\theta_A$, $\nu'\le\nu_{CL}$ & $k_D(n)\ge78$; \eqref{eq:CWconst} \\
Prop.~\ref{prop:D} & $n\ge10^5$, $k\le1.7n^{1/3}$ & $\Psi_D$ increasing \\
Prop.~\ref{prop:rep} & $2\le k\le n-2$ & $79\le k\le\lceil n/2\rceil$ on $CL\cup W\cup CR$ \\
\bottomrule
\end{tabular}
\end{table}

%======================================================================
\section{Finite verification for \texorpdfstring{$n\le2\cdot10^5$}{n <= 2*10\^5}}\label{sec:finite}
%======================================================================

By Lemma~\ref{lem:E}, the row $n$ is unimodal if and only if the signs of $p(n,k)-p(n,k-1)$ for $2\le k\le\lceil n/2\rceil+1$ contain no
$+$ after a $-$. All $p(n,k)$ with $n\le N$ are produced simultaneously by the column recurrence
\[
  v[j]\leftarrow v[j]+v[j-k]\qquad(j=k,k+1,\dots),\qquad k=1,2,\dots,
\]
started from $v=(1,0,0,\dots)$: after step $k$, $v[j]$ is the number of partitions of $j$ into parts $\le k$, which equals $p(j+k,k)$
by Lemma~\ref{lem:comb}. The cost is $O(N^2)$ additions.

Two independent programs implement this.
\begin{enumerate}
\item \emph{Exact integers} (\script{ws\_fc\_exact.c}). Fixed-width multi-limb unsigned integers ($27$ limbs of $64$ bits for $N=2\cdot10^5$,
  justified by $p(N)<e^{\pi\sqrt{2N/3}}$; a carry beyond the last limb aborts the run). Only additions and comparisons are performed.
  Since $v[j+1]$ is not yet updated when $v[j]$ has just been updated at step $k$, the comparison of $p(n,k)$ with $p(n,k-1)$ needs no
  second array. The program uses no floating point and no threads.
\item \emph{Interval arithmetic} (\script{ws\_fc\_check.c}). Two copies of the array in x87 extended precision ($64$-bit significand,
  no overflow since $\log_2p(2\cdot10^5)\approx1660$), one updated with rounding toward $-\infty$ and one toward $+\infty$. Since every
  operation is an addition of nonnegative numbers, induction gives $\mathrm{lo}[j]\le v[j]\le\mathrm{hi}[j]$ throughout. A pair is classified
  as $+$ or $-$ only if the intervals are disjoint; no pair remained unclassified.
\end{enumerate}
For $N=2\cdot10^5$ both programs report no violation; there are $256\,705\,214$ increasing pairs, $9\,743\,394\,777$ decreasing pairs and $8$
ties (all for small $n$). They write a per-$n$ log (mode and verdict) and the two logs are byte-identical. For $N=2000$ both logs also agree
byte-for-byte with a third, Python big-integer computation via $p(n,k)=p(n-1,k-1)+p(n-k,k)$. As an illustration, the mode of the row is
$k=2042$ for $n=2\cdot10^5$.

Since $2\cdot10^5>10^5$, Sections~\ref{sec:assembly} and~\ref{sec:finite} together prove the Main Theorem for every $n\ge1$.

\begin{theorem}[Strict form]\label{thm:strict}
For every $n\ge4$ there is an index $t=t(n)$ with
$p(n,1)<\dots<p(n,t)$, $p(n,t)\ge p(n,t+1)$, $p(n,t+1)>p(n,t+2)>\dots>p(n,n-1)$ and $p(n,n-1)=p(n,n)=1$.
Moreover $p(n,t)=p(n,t+1)$ exactly for $n\in\{5,6,8,13,14\}$.
\end{theorem}

\begin{proof}
For $n\le2\cdot10^5$ both programs classify every pair $(p(n,k),p(n,k+1))$ with $k\le\lceil n/2\rceil$ as increasing, decreasing or tied.
They report exactly $8$ ties, and a direct computation for $n\le200$ locates all eight:
\[
(n,k)=(2,1),\ (3,1),\ (3,2),\ (5,2),\ (6,2),\ (8,3),\ (13,4),\ (14,4),
\]
each between the two largest terms of its row. Lemma~\ref{lem:E}
gives strict decrease for $\lceil n/2\rceil\le k\le n-2$. For $n\ge10^5$, every comparison is strict in Regions $D$, $CL$ and $CR$
(Section~\ref{sec:CLCR}) and on $W$ outside the at most one undetermined index of Corollary~\ref{cor:onek}; that index lies at
the sign change, so a tie can only occur between the two largest terms. Region $E$ is handled by Lemma~\ref{lem:E}.
\end{proof}

%======================================================================
\section{The computer-assisted parts}\label{sec:computer}
%======================================================================

\subsection{Method}
All inequalities in the analytic part that involve explicit numbers are certified in Arb~\cite{Arb} ball arithmetic (python-flint 0.9.0).
A ball enclosure of a function on a box contains every value the function takes on that box, so a strict inequality certified on each box of
a cover holds on the union. Where a statement is about all $\theta$ near $0$ or all $\beta\ge\beta_A$, it is reduced to finitely many boxes
either by a tail box in variables that are monotone on the tail (Lemma~\ref{lem:W}), or by a monotone majorant whose monotonicity is checked
symbolically (Section~\ref{sec:proofA}, Lemma~\ref{lem:C1}). Floating-point grids are used only to choose the cells. Integrals are bounded
below by right Riemann sums of decreasing integrands. Exact algebraic identities (Gaussian moments, the polynomial identity of
Lemma~\ref{lem:gap}, derivatives, the Dirichlet identity~\eqref{eq:dirichlet}, the telescoping in Lemma~\ref{lem:Fr}) are checked with
SymPy.

\subsection{List of certified statements}
Table~\ref{tab:certs} lists every computer-assisted statement used in the proof, the program that certifies it (in the directory
\texttt{analytic/} or \texttt{finite/} of the repository), and the certified value. The driver script \texttt{run\_all.sh} runs all
the analytic programs and the exact finite check (to $2\cdot10^5$ with the option \texttt{--full}) and compares each output with a
stored reference; the interval program F2 is run separately (see the repository README); on a two-core machine the analytic part takes about five minutes and the
exact finite check to $2\cdot10^5$ about eight minutes.

{\small
\begin{longtable}{@{}>{\raggedright\arraybackslash}p{0.07\textwidth}>{\raggedright\arraybackslash}p{0.25\textwidth}>{\raggedright\arraybackslash}p{0.36\textwidth}>{\raggedright\arraybackslash}p{0.22\textwidth}@{}}
\caption{Computer-assisted statements. ``Sanity'' rows test statements that are proved in the text and are not part of the proof.}\label{tab:certs}\\
\toprule
Item & Program & Statement certified & Value \\
\midrule
\endfirsthead
\toprule
Item & Program & Statement certified & Value \\
\midrule
\endhead
\bottomrule
\endfoot
S1 & \script{ws\_mech\_elem} & elementary monotonicity facts; Dirichlet identity~\eqref{eq:dirichlet}; Abel summation; $uf_r'=rf_r-f_{r+1}$ & exact (SymPy) \\
S2 & \script{ws\_mech\_lemA} & Gaussian moments; $D_0=\sqrt{2\pi/B}N_0$; $\tilde N_0$ real; Lemma~\ref{lem:gap}; scaling invariance; the decompositions in Lemma~\ref{lem:central}; algebraic reduction to $G$ in Lemma~\ref{lem:W}; independent re-implementation of $\varepsilon$ agrees with the production code & exact; agreement $1.7\cdot10^{-13}$ \\
S3 & \script{ws\_close\_elem} & telescoping in Lemma~\ref{lem:Fr} and its limits; Arb quadrature of $F_r$; identities in Lemma~\ref{lem:B1}; Gaussian integrals in Lemmas~\ref{lem:C0},~\ref{lem:delta}, Prop.~\ref{prop:CLa} & exact / Arb \\
C1 & \script{ws\_mech\_window} & Lemma~\ref{lem:W} and~\eqref{eq:CWconst} for all $\beta\ge\beta_A$ & $G\ge1.0092$, $\theta_A\le0.0040939$, $\nu'\in[-2.5102,2.0840]$, $\nu\ge3.1158$, $\nu_{CL}\le-1.95361$, $\nu_{CR}\ge1.04305$ \\
C2 & \script{ws\_mech\_tv} & Lemma~\ref{lem:TV} & see Lemma~\ref{lem:TV} \\
C3 & \script{ws\_mech\_minor} & Corollary~\ref{cor:cstar} & $c_\ast\ge0.15973$ \\
A1 & \script{ws\_p12\_boxes} & Theorem~\ref{thm:A}, $\theta\in[0.002,\theta_A]$ & $\Gamma\le0.42024$ \\
A2 & \script{ws\_p12\_asym} & Theorem~\ref{thm:A}, $\theta\le0.002$ & $\Gamma\le0.17605$ \\
B1 & \script{ws\_b\_step} & Theorem~\ref{thm:B}, $\theta\in[0.002,\theta_A]$ & $\Sigma\ge0.91918$ \\
B2 & \script{ws\_b\_asym} & Theorem~\ref{thm:B}, $\theta\le0.002$ & $\Sigma\ge0.96039$ \\
B3 & \script{ws\_b\_edges}, \script{ws\_mech\_edges} & edge values of $L_2/\theta$ (not needed for unimodality) & $\ge5.588$; $\le-0.505$ \\
M & \script{ws\_mech\_mono} & monotonicity in $\theta$ of the majorants in Lemma~\ref{lem:C1} and Section~\ref{sec:proofA} & exact (SymPy) \\
R1 & \script{ws\_mech\_cd}, \script{ws\_cd\_CR} & Proposition~\ref{prop:CR} & $0.3852<1$ \\
R2 & \script{ws\_mech\_cla}, \script{ws\_cd\_CLa} & Proposition~\ref{prop:CLa} & $\rho\le0.0673<0.8976$ \\
R3 & \script{ws\_mech\_cd}, \script{ws\_cd\_CLb} & Proposition~\ref{prop:CLb} & min.\ ratio $5.82>1$ \\
R4 & \script{ws\_close\_D}, \script{ws\_cd\_D} & Proposition~\ref{prop:D}: $\partial_k\Phi$, $\Psi_D'$ in closed form, $\Psi_D(s_A)>0$ & $\Psi_D(s_A)=0.30456$ \\
T & \script{ws\_close\_thresh} & Table~\ref{tab:thresholds} and the cover in Section~\ref{sec:assembly}, reading the values certified above & --- \\
F1 & \script{ws\_fc\_exact.c} & Section~\ref{sec:finite}, exact integers, $n\le2\cdot10^5$ & 0 violations \\
F2 & \script{ws\_fc\_check.c} & Section~\ref{sec:finite}, directed-rounding intervals, $n\le2\cdot10^5$ & 0 violations, identical log \\
\midrule
Sanity & \script{ws\_close\_rep} & \eqref{eq:rep1} in exact rational arithmetic ($4959$ cases), \eqref{eq:rep2} by quadrature, Lemmas~\ref{lem:psider}, \ref{lem:Hder}, \ref{lem:D1}, \ref{lem:E} on instances & --- \\
Sanity & \script{ws\_b\_check}, \script{ws\_p12\_window}, \script{ws\_cd\_window}, \script{ws\_p12\_minor}, \script{ws\_b\_minorM} & pointwise cross-checks of the above with independent code & --- \\
\end{longtable}
}

\subsection{Trust base}
The proof relies on: the correctness of Arb's ball arithmetic as exposed by python-flint 0.9.0 (including its polylogarithm, zeta, gamma,
error-function and integration routines); SymPy 1.14 for exact symbolic identities; the GCC 13.3 compiler for the integer code of the finite
check (cross-checked by the independent interval program and, for $n\le2000$, by Python big integers); and the following classical results,
which are quoted rather than proved: Taylor's theorem with integral remainder, the mean value theorem, the Weierstrass M-test for termwise
differentiation, and additivity of total variation. All other steps are proved in this paper.

%======================================================================
\appendix
\section{The error bound in scaled form}\label{app:eps}
%======================================================================

We record the expression for $\varepsilon_k$ of Proposition~\ref{prop:eps} as it is evaluated. Set $B=1$ and replace $\kappa_r,c_r$ by
$A_r,C_r$ of~\eqref{eq:AC} and $\varphi_0$ by $w_0=0.4\sqrt{b/\theta}$. Write $a_r=A_r/r!$, $\delta_0=A_4w_0^2/12+A_6w_0^4/360$ and
$B'=1-\delta_0$. For a polynomial $P(t)=\sum p_jt^j$ with $p_j\ge0$ and $A>0$ put
\[
  \mathcal I_A[P]:=\frac1{\sqrt{2\pi}}\int_\RR P(|w|)e^{-Aw^2/2}\dd w=\sum_jp_j\,\frac{2^{j/2}\Gamma(\frac{j+1}2)}{\sqrt\pi}\,A^{-(j+1)/2}.
\]
Let $Q_{\mathrm{abs}}=1+a_3t^3+a_4t^4+a_5t^5+\frac12a_3^2t^6+a_3a_4t^7+\frac16a_3^3t^9$ and
$P_{\mathrm{abs}}=C_1t+\frac12(C_2+C_1^2)t^2+\frac16C_3t^3+\frac1{24}C_4t^4$. Then
\[
\begin{aligned}
E_N&=\mathcal I_{B'}[C_1t\,T_\omega]+\mathcal I_1[Q_{\mathrm{abs}}T_H]+e^{-w_0^2/4}\,\mathcal I_{1/2}[Q_{\mathrm{abs}}P_{\mathrm{abs}}]+2\mathfrak m,\\
E_D&=\mathcal I_{B'}[T_\omega]+e^{-w_0^2/4}\,\mathcal I_{1/2}[Q_{\mathrm{abs}}]+\mathfrak m,
\end{aligned}
\]
where $T_\omega,T_H$ are as in Section~\ref{sec:main} with $(\kappa_r,c_r)\to(A_r,C_r)$, and
\[
  \mathfrak m:=2\pi\,e^{-c_\ast/\theta}\Bigl(\frac{b}{2\pi\theta^3}\Bigr)^{1/2}
\]
is the minor-arc term $2\pi M_\chi/\sqrt{2\pi/B}$ with Corollary~\ref{cor:cstar}. (Since $\mathcal I_A$ integrates against $e^{-Aw^2/2}$, the central-arc terms carry the factor $B'^{-1/2}$ and the tail terms the
factor $\sqrt2$ automatically; the code evaluates $\mathcal I_A[P]=A^{-1/2}\,\E P(|Z_A|)$ with $Z_A\sim N(0,1/A)$.) With $N_0$, $M_2$ and
$\mathrm{Num}$ in scaled form,
\[
  \Delta=\frac{E_N+|\tilde N_0/D_0|E_D}{N_0-E_D}+\frac{|\mathrm{Num}|}{N_0},\qquad \varepsilon_k=-\log(1-\Delta/M_2).
\]
With the orders of magnitude noted after Lemma~\ref{lem:gap}, $E_N=O(\theta^{7/2})$, $E_D=O(\theta^2)$, $\tilde N_0/D_0=O(\theta^2)$ and
$\mathrm{Num}=O(\theta^4)$ (the tails and minor arcs being exponentially small), so $\varepsilon_k=O(\theta^{7/2})$, which is small compared
with $\theta q\asymp\theta^2$ for $\nu'$ in a bounded range. Theorem~\ref{thm:A} quantifies this.

%======================================================================
\section{The monotone majorant for \texorpdfstring{$\theta\le0.002$}{theta <= 0.002}}\label{app:regionII}
%======================================================================

Let $\theta_1=0.002$, $t=\sqrt\theta\le t_1=\sqrt{\theta_1}$, and fix a cell $[v_0,v_0']$ of $\nu'$. By~\eqref{eq:AC} and
$q=\theta e^{-\nu'}=t^2e^{-\nu'}$:
\begin{itemize}
\item $A_r=s_rb^{-r/2}t^{r-2}$, where $s_r\le r!\zeta_2+\theta_1\TV(f_r)$ and $b\ge F_2(\nu_{\min}-\theta_1)-\theta_1\TV(f_2)$;
\item $C_r=f_r(\nu)b^{-r/2}t^r$ with $f_r(\nu)=\nu^re^{-\nu}P_{r-1}(e^{-\nu})/(1-e^{-\nu})^r\le\nu^r\,q\,P_{r-1}(y)/(1-y)^r$,
  $y=e^{-\nu_{\min}}$ (Eulerian polynomials $P_{r-1}$ have nonnegative coefficients);
\item $C_1=c_1(\theta^3/b)^{1/2}$ with $c_1=1+\nu e^{-\nu'}/(1-q)$.
\end{itemize}
Each term of $E_N$, $E_D$, $\mathrm{Num}$ and $M_2-1$ is a polynomial with nonnegative coefficients in these quantities, except for the
explicitly signed terms of $M_2$ and $N_0$, which are bounded below termwise (for instance $N_0\ge1-\frac5{24}A_3^2$). Expanding, each term
is a nonnegative coefficient times $t^p\,q^{b_q}\,\nu^j$. After division by $\theta q=t^2q$ the monomial becomes
$t^P\nu^je^{-(b_q-1)\nu'}$, and at fixed $\nu'$, $\nu=\nu'+2\log(1/t)$. By the identity in Section~\ref{sec:proofA}, $t\mapsto t^P\nu^j$ is
nondecreasing on $(0,t_1]$ when $P>0$ and $\nu_{\min}>2j/P$ (for $j<0$ it is increasing whenever $P>0$). The program asserts these two
conditions for every non-constant monomial, evaluates each at $t=t_1$ with $\nu$ at the appropriate end ($\nu\le v_0'+\log(1/\theta_1)$ if $j\ge0$,
$\nu\ge\nu_{\min}$ if $j<0$), and bounds $e^{-(b_q-1)\nu'}$ at the appropriate end of the cell. The non-polynomial factors
$e^{-w_0^2/4}=e^{-0.04b/\theta}$ and $e^{-c_\ast/\theta}$ are handled by
$\frac{d}{d\theta}\log(\theta^{-p}e^{-a/\theta})=(a-p\theta)/\theta^2$, which is positive for $\theta<a/p$. The same method, applied to
the increments of Lemma~\ref{lem:B1} at the two points $k$ and $k+1$, gives Region II of Theorem~\ref{thm:B}.

\subsection*{Acknowledgments}
AI tools (Anthropic's Claude) were used in this work, including for the computations, the accompanying code, and the preparation
of the manuscript. The author takes full responsibility for the content, and all computational claims can be reproduced from the
public repository.

%======================================================================

\end{document}